%% file: liegroups.tex
\documentclass[12pt]{article}

\input{Setup_handout.tex}
\input{Commands.tex}
\input{Environments.tex}

\title{Lie groups of Poisson diffeomorphisms}
\author{Wilmer Smilde\footnote{wsmilde2@illinois.edu}}
\date{}

\begin{document}
\maketitle
\begin{abstract} By considering suitable Poisson groupoids, we develop an approach to obtain Lie group structures on (subgroups of) the Poisson diffeomorphism groups of various classes of Poisson manifolds. As applications, we show that the Poisson diffeomorphism groups of (normal-crossing) log-symplectic, elliptic symplectic, scattering-symplectic and cosymplectic manifolds are regular infinite-dimensional Lie groups.
\end{abstract}
\tableofcontents

\input{0Introduction.tex}

\input{1Poissondiffeomorphisms.tex}

\input{3LiegroupsofPoissondiffeomorphisms}

\appendix

 \small
\bibliographystyle{abbrv}

\end{document}

%% file: Setup_handout.tex
\usepackage[utf8]{inputenc}
\usepackage[T1]{fontenc}
\usepackage{amsmath, amssymb, amsthm, amsfonts, mathtools, tikz, tikz-cd, xcolor, hyperref, enumitem, mathrsfs}
\usepackage[nottoc]{tocbibind}
\usepackage[a4paper, margin=1.1 in]{geometry}
\usepackage{multicol}
\usepackage{titlesec}

 \usepackage{libertine} 
 \usepackage[libertine]{newtxmath}

\titleformat{\subsection}[runin]{\bfseries}{\thesubsection}{1em}{}[.]
\titleformat{\subsubsection}[runin]{\itshape}{\thesubsubsection}{1em}{}[.]

%% file: Commands.tex
\newcommand{\RR}{\mathbb{R}}

\newcommand{\LL}{\mathcal{L}}
\renewcommand{\AA}{\mathcal{A}}
\newcommand{\BB}{\mathcal{B}}
\newcommand{\VV}{\mathcal{V}}

\newcommand{\GG}{\mathcal{G}}
\newcommand{\toto}{\rightrightarrows}
\newcommand{\rar}[1]{\overrightarrow{#1}}
\newcommand{\lar}[1]{\overleftarrow{#1}}

\newcommand{\del}{\partial}
\newcommand{\deldel}[1]{\frac{\del}{\del {#1}}}
\newcommand{\can}{\mathrm{can}}

\newcommand{\FF}{\mathcal{F}}
\newcommand{\EE}{\mathcal{E}}

\newcommand{\mc}[1]{\mathcal{#1}}

\newcommand{\mf}[1]{\mathfrak{#1}}

\let\Im\relax

\DeclareMathOperator{\Im}{Im}

\DeclareMathOperator{\id}{Id}

\DeclareMathOperator{\Diff}{Diff}

\DeclareMathOperator{\Bis}{Bis}
\DeclareMathOperator{\Pair}{Pair}

\DeclareMathOperator{\Fol}{Fol}
\newcommand{\fol}{\mf{fol}}
\DeclareMathOperator{\Ham}{Ham}
\newcommand{\ham}{\mf{ham}}
\newcommand{\Hamloc}{\operatorname{Ham}_{\mathrm{loc}}}
\newcommand{\hamloc}{\mf{ham}_{\mathrm{loc}}}
\newcommand{\Hamlocc}{\operatorname{Ham}_{\mathrm{loc,c}}}
\newcommand{\hamlocc}{\mf{ham}_{\mathrm{loc},c}}
 
\DeclareMathOperator{\InnAut}{InnAut}
\DeclareMathOperator{\Hol}{Hol}

\makeatletter
\newsavebox{\@brx}
\newcommand{\llangle}[1][]{\savebox{\@brx}{\(\m@th{#1\langle}\)}%
	\mathopen{\copy\@brx\kern-0.5\wd\@brx\usebox{\@brx}}}
\newcommand{\rrangle}[1][]{\savebox{\@brx}{\(\m@th{#1\rangle}\)}%
	\mathclose{\copy\@brx\kern-0.5\wd\@brx\usebox{\@brx}}}
\makeatother

%% file: Environments.tex
\theoremstyle{plain}
\newtheorem{theorem}{Theorem}[section]
\newtheorem{proposition}[theorem]{Proposition}
\newtheorem{corollary}[theorem]{Corollary}
\newtheorem{lemma}[theorem]{Lemma}

\newtheorem{question}[theorem]{Question}

\theoremstyle{definition}
\newtheorem{df}[theorem]{Definition}
\newtheorem{ex}[theorem]{Example}
\newtheorem{remark}[theorem]{Remark}

\newenvironment{definition}
  {\pushQED{\qed}\df}
  {\popQED\enddf}
\newenvironment{example}
  {\pushQED{\qed}\ex}
  {\popQED\endex}

\newenvironment{claimproof}[1]{\noindent\textit{Proof of claim.}\space#1}{\par\addvspace{8pt}}

\newtheorem*{claim*}{Claim}
\newtheorem*{notation*}{Notation}

%% file: 0Introduction.tex
\section*{Introduction}\addcontentsline{toc}{section}{Introduction}
\newcommand{\Marcut}{M\u{a}rcu\textcommabelow{t}}

A Poisson structure on a manifold $M$ is a Lie bracket $\{\cdot, \cdot\}$ on the space of smooth functions $C^\infty(M)$ subject to the Leibniz rule $\{fg, h\}=f\{g, h\}+g\{f, h\}$ for $f,g,h \in C^\infty(M)$. Alternatively, the Poisson structure can be described by a bivector field $\pi\in \Gamma(\wedge^2TM)$ for which the Schouten-Nijenhuis bracket $[\pi, \pi]$ vanishes. It induces the Lie bracket on $C^\infty(M)$ via $\{f, g\}=\pi(df, dg)$. 

A \textit{Poisson map} $f:(M, \pi_M)\to (N, \pi_N)$ between Poisson manifolds $(M, \pi_M)$ and $(N, \pi_N)$ is a smooth map $f$ for which $\pi_M$ is $f$-related to $\pi_N$, meaning that $\wedge^2 T_x f (\pi_{M, x})=\pi_{N, f(x)}$ for all $x\in M$. 

A Poisson manifold $(M, \pi)$ comes with a symplectic foliation, spanned by the \textit{Hamiltonian vector fields}
\[
X_f=\{f, \cdot\}, \quad \mbox{for $f\in C^\infty(M)$}.
\]
A symplectic leaf $S$ comes with a symplectic stucture determined by $\omega_S(X_f, X_g)=-\{f, g\}\vert_S$ for $f, g\in C^\infty(M)$.

A \textit{Poisson diffeomorphism} of $(M, \pi)$ is a diffeomorphism of $M$ that is also a Poisson map. This defines a subgroup $\Diff(M,\pi)$ of $\Diff(M)$, called the Poisson diffeomorphism group. Note that a diffeomorphism is Poisson if and only if it sends each symplectic leaf symplectomorphically onto a (possibly different) symplectic leaf.
\subsection*{The Poisson diffeomorphism group}

The group of symplectomorphisms $\Diff(M, \omega)$ of a symplectic manifold $(M, \omega)$ is well-understood  as an infinite-dimensional Lie group \cite{ebinmarsden1970, krieglmichor1997}. Much less is known about $\Diff(M, \pi)$ for general Poisson manifolds.

For an arbitrary Poisson manifold $(M, \pi)$, the group of Poisson diffeomorphisms $\Diff(M, \pi)$ can have many interesting subgroups:
\begin{itemize}[noitemsep, topsep=0em]
	\item \textit{Foliated Poisson diffeomorphisms} $\Fol(M, \pi)$ are those that send each leaf to itself. Correspondingly, one has the \textit{foliated Poisson vector fields} $\fol(M, \pi)$, which are the Poisson vector fields that are tangent to the symplectic foliation.  
	\item The group of \textit{Hamiltonian diffeomorphisms} $\Ham(M, \pi)$ consists of time-1 flows of Hamiltonian vector fields generated by a time-dependent Hamiltonian function.
	\item The \textit{locally Hamiltonian diffeomorphisms} $\Hamloc(M, \pi)$ are the of time-1 flows of locally Hamiltonian vector fields generated by a time-dependent closed one-form.
\end{itemize}
\begin{example}[Symplectic manifolds]
For a compact symplectic manifold $(M, \omega)$, the group of locally Hamiltonian diffeomorphisms coincides with the identity component of the symplectomorphism group. This follows from the fact that every symplectic vector field is locally Hamiltonian. 
\end{example}
\begin{example}[Zero Poisson manifolds]
	For the zero Poisson manifold $(M,0)$, all of the above subgroups are trivial, while $\Diff(M, \pi)=\Diff(M)$. 
\end{example}
\begin{example}[Poisson manifolds of constant rank]
	Let $(M, \pi)$ be a Poisson manifold for which the bivector $\pi$ has constant rank, with underlying symplectic foliation $(\FF, \omega)$. In general, $\Diff(M, \pi)$ is larger than $\Fol(M, \pi)$, and $\Hamloc(M, \pi)$ is smaller than (the identity-component of) $\Fol(M,\pi)$. The latter can be seen, if they were Lie groups, on the level of Lie algebras. The Lie algebra $\mf{fol}(M, \pi)$ is isomorphic via $\omega^\flat$ to the space $\Omega^1_{\mathrm{cl}}(\FF)$ of closed \textit{foliated} one-forms, while $\Hamloc(M, \pi)$ is isomorphic to the space of foliated one-forms that admit a closed extension in $\Omega^1(M)$. 
\end{example}

In this paper, we investigate the existence of Lie group structures on the group of Poisson diffeomorphisms and on its subgroups for several classes of Poisson manifolds.

\subsection*{Coisotropic bisections of Poisson groupoids}

When $(M, \omega)$ is a symplectic manifold, a manifold structure on $\Diff(M, \omega)$ can be obtained in the following way. First, symplectomorphisms on $(M, \omega)$ correspond to Lagrangian submanifolds of $(M\times M, \rar{\omega}-\lar{\omega})$. Second, by means of Weinstein's Lagrangian neighbourhood theorem, the symplectic form $\rar{\omega}-\lar{\omega}$ can be linearized around the graph of a symplectomorphism. Finally, the linearization of $\omega$ around the graph also linearizes the deformation space of the Lagrangian submanifold: the Lagrangian submanifolds near the graph of a symplectomorphism form a closed subspace of the sections of the normal bundle. This `models' the symplectomorphism group on an infinite-dimensional locally convex space, namely $\Omega^1_{\mathrm{cl}, c}(M)$, the space of compactly supported closed one-forms on $M$ (see \cite{krieglmichor1997}, section 43, for more details). 

For a general Poisson manifold $(M, \pi)$, things become much more complicated, partly because of the intransitive nature on both the global and local level of the Poisson diffeomorphism groups. For example, points on non-symplectomorphic leaves can not be permuted, as any Poisson diffeomorphism sends each symplectic leaf symplectomorphically onto another leaf.

A useful framework to capture intransitive symmetries in differential geometry is provided by Lie groupoids and Lie algebroids. The group of bisections $\Bis(\GG)$ of a Lie groupoid $\GG\toto M$ naturally acts on $M$. Moreover, it is well-known that $\Bis(\GG)$ comes with a natural Lie group structure \cite{Rybicki2002, SchmedingWockel2015}. We will see in many examples that the bisection group of an appropriate Lie groupoid corresponds to an intransitive subgroup of $\Diff(M)$. 

When $(\GG, \Pi)\toto (M, \pi)$ is a Poisson groupoid \cite{weinstein1988}, the group of coisotropic bisections $\Bis(\GG, \Pi)$ acts by Poisson diffeomorphisms on $(M, \pi)$. This puts the object $\Diff(M, \pi)$ in a more general perspective: coisotropic bisections of the \textit{pair Poisson groupoid} $(M\times M, \lar{\pi}-\rar{\pi})\toto (M, \pi)$ are in one-to-one correspondence with Poisson diffeomorphisms of $(M, \pi)$.

In \cite{marcut2020}, \Marcut{} constructs an example of a compact Poisson manifold of constant rank for which $\Diff(M, \pi)$ is not locally path-connected in $\Diff(M)$. This is on itself not surprising or any problem --it only tells that $\Diff(M, \pi)$ can not be embedded in $\Diff(M)$-- but it does indicate that the pair Poisson groupoid might not be the right object to investigate, and we must broaden our perspective to involve other Poisson groupoids. However, in general there is no known or obvious Lie group structure on the group of coisotropic bisections. 
\begin{example}[Lagrangian bisections]\label{ex:introlagrangianbisections}
	A particularly important class of Poisson groupoids over a Poisson manifold $(M, \pi)$ are, when integrable, the symplectic groupoids. Let $(\GG, \Omega)\toto (M, \pi)$ be a symplectic groupoid. Its coisotropic bisection are better known as Lagrangian bisections, studied in \cite{Rybicki2001, xu1997}. As explained in more detail in Section \ref{sec:lagrangianbisections}, the group of Lagrangian bisections $\Bis(\GG, \Omega)$ has interesting but complicated interactions with the Poisson diffeomorphism group $\Diff(M, \pi)$. For instance, the identity path component $\Bis_{c, 0}(\GG, \Omega)$ acts on $(M, \pi)$ by locally Hamiltonian diffeomorphisms! In fact, any locally Hamiltonian diffeomorphism comes from a Lagrangian bisection. The map $\Bis_{c, 0}(\GG, \Omega)\to \Diff(M, \pi)$ has a non-trivial kernel in general. 
\end{example}

\subsection*{The Lie group of coisotropic bisections}Let $(\GG,\Pi)\toto (M, \pi)$ be a Poisson groupoid. In order to build charts for a manifolds structure on $\Bis(\GG, \Pi)$, it is necessary to describe the coisotropic deformations of the unit manifold $M\subset (\GG, \Pi)$. It turns out $M$ is always \textit{Lagrangian} in $(\GG, \Pi)$ (Proposition \ref{prop:coisotropicbisections}), in which case it makes sense to ask whether the Poisson structure $\Pi$ is linearizable around it. That is: does there exists a local Poisson diffeomorphism $(\GG, \Pi)\dashrightarrow (NM, \Pi_{\mathrm{lin}})$, that restricts to the identity on $M$? Here, the Poisson structure $\Pi_{\mathrm{lin}}$ is the linearization of $\Pi$ on the normal bundle $NM$ of $M$ in $\GG$, via any tubular neighbourhood (see section \ref{subsec:linearpoissonstructures}). If the answer to this question is positive, we say that the Poisson groupoid $(\GG, \Pi)$ is \textit{linearizable}.

Under the additional assumption of linearizability of a Poisson groupoid, we show that its coisotropic bisection group is a Lie group. The used framework for infinite-dimensional manifolds is the "convenient setting" by Kriegl and Michor \cite{krieglmichor1997}. Our first main result is the following.
\begin{theorem}[Theorem \ref{thm:coisotropicbisectionsliegroup}]\label{thm:introcoisotropicbisections}
	Let $(\GG, \Pi)\toto (M,\pi)$ be a Poisson groupoid with Lie bialgebroid $(\AA, \AA^*)$. Assume that $(\GG, \Pi)$ is linearizable around $M$. Then $\Bis(\GG, \Pi)$ is a regular embedded Lie subgroup of $\Bis(\GG)$ with Lie algebra $\Gamma_c(\AA, d_{\AA^*})=\{ v\in \Gamma_c(\AA): d_{\AA^*} v=0\}$, where $d_{\AA^*}$ is the differential associated to the Lie algebroid $\AA^*\Rightarrow M$.
\end{theorem}

This breaks up the search of Lie group structures on Poisson diffeomorphism groups into two steps, which are both very interesting on their own.
\begin{itemize}[noitemsep, topsep=0em, leftmargin=2cm]
	\item[Step 1.] Find out which Poisson groupoids are linearizable.
	\item[Step 2.] Search for Poisson groupoids whose coisotropic bisections `compute' relevant groups of Poisson diffeomorphisms on the base.  
\end{itemize}
In fact, to apply Theorem \ref{thm:introcoisotropicbisections}, one is often forced to look for Poisson groupoids beyond the product, because of the following result from \cite{smilde2021linearization}.
\begin{theorem}[\cite{smilde2021linearization}, Theorem 3.26]
	Let $(M, \pi)$ be a Poisson manifold. If $(M\times M, \lar{\pi}-\rar{\pi})$ is linearizable around the diagonal, then $\pi$ has constant rank. 
\end{theorem}
The linearization problem of Poisson groupoids has been investigated in \cite{smilde2021linearization}, where it is shown that integrations of triangular Lie algebroids of so-called cosymplectic type, as well as dual integrations of arbitrary triangular Lie bialgebroids are always linearizable. We recall these results in more detail in Section \ref{sec:linearization}.

Building on this work, we show in this paper that several classes of Poisson manifolds come with relevant Poisson groupoids that turn out to be linearizable. This way, we obtain Lie group structures on their groups of Poisson diffeomorphisms. The results are collected in the list of examples at the end of the introduction.

\subsection*{Lie groups of Poisson diffeomorphisms}
When a Poisson structure $\pi$ on $M$ `lifts' (in the sense of the diagram below) to a Poisson structure $\pi_\AA\in \Gamma(\wedge^2\AA)$ on a Lie algebroid $\AA\Rightarrow M$, integrations of $\AA$ are particularly relevant to the group of Poisson diffeomorphisms. 
\[
\begin{tikzcd}
	\AA^* \arrow[r, "\pi_\AA^\sharp"] &  \AA \arrow[d]\\
	T^*M \arrow[u] \arrow[r, "\pi^\sharp"] & TM
\end{tikzcd}
\]
We mainly focus on the case that $\AA$ is an almost injective Lie algebroid, in which case we obtain two results of a different kind. The first is about inner (Poisson) automorphisms, and the second one involves the ($\AA$-)locally Hamiltonian diffeomorphisms. 
\subsubsection*{Inner automorphisms of almost injective Lie algebroids} A Lie algebroid $\AA\Rightarrow M$ is almost injective when the anchor map is injective on the level of sections. In this case, the bisections of the holonomy groupoid of $\AA$ \cite{androulidakiszambon2017, debord2000} compute the inner automorphisms of $\AA$. Any automorphism of an almost injective Lie algebroid is determined by its base map, so the inner automorphisms $\InnAut(\AA)$ form a subgroup of $\Diff(M)$. These observations combine to the following.  
\begin{theorem}[Corollary \ref{cor:innautliesubgroup}]\label{thm:introinneraut}
	Let $\AA\Rightarrow M$ be an almost injective Lie algebroid. Then the group $\InnAut_c(\AA)$ is naturally a Lie group with Lie algebra $\Gamma_c(\AA)$ for which the inclusion $\InnAut_c(\AA)\to \Diff(M)$ is a smooth immersion. 
\end{theorem}
	
A \textit{k-cosymplectic structure} on a Lie algebroid $\AA\Rightarrow M$ consists of closed one-forms $\alpha_1, \dots, \alpha_k\in \Omega^1(\AA)$ and a closed 2-form $\omega\in \Omega^2(\AA)$ such that $\alpha_1\wedge\dots\wedge\alpha_k\wedge\omega\neq 0$ and $\AA=\FF\oplus \ker \omega$, with $\FF=\cap_i\ker \alpha_i$. The subbundle $\FF\subset \AA$ is involutive and thus can be regarded as a subalgebroid. Note that $\omega\vert_\FF$ is non-degenerate, and thus induces an $\AA$-Poisson structure $\pi_\AA$ via
\[
\begin{tikzcd}
	\FF^* & \FF \arrow[l, "\left(\omega\vert_\FF\right)^\flat"'] \arrow[d] \\
	\AA^* \arrow[u] \arrow[r, "\pi_\AA"] & \AA.
\end{tikzcd}
\]
We call an $\AA$-Poisson structure $\pi_\AA$ of ($k$-)cosymplectic type if it is induced by some ($k$-)cosymplectic structure on $\AA$.	
\begin{theorem}[Theorem \ref{thm:almostinjectivePoisson}]
	Let $(\AA, \pi_\AA)\Rightarrow M$ be a Lie algebroid with an $\AA$-Poisson structure of cosymplectic type. Then $\InnAut_c(\AA, \pi_\AA)$ is an embedded Lie subgroup of $\InnAut_c(\AA)$ with Lie algebra $\Gamma_c(\AA, \pi_\AA)=\left \{ v\in \Gamma_c(\AA):[\pi_\AA, v]=0\right \}$.
\end{theorem}
Under additional assumptions, the groups $\InnAut_c(\AA)$ and $\InnAut_c(\AA, \pi_\AA)$ are initial in $\Diff(M)$. Concretely, this means that a curve $\varphi_t$ in $\InnAut_c(\AA)$ is smooth if and only if the map $(t, x)\mapsto \varphi(t, x)$ is. 

\subsubsection*{Locally Hamiltonian diffeomorphisms} Recall from example \ref{ex:introlagrangianbisections} that the group of Lagrangian bisections of a symplectic groupoid surjects onto the group of locally Hamiltonian diffeomorphisms of its underlying Poisson manifold. In case that $\pi$ is non-degenerate almost everywhere, then the surjection has a discrete kernel. This lays at the basis of the proof of the following result. 
\begin{theorem}[Theorem \ref{thm:hamlocinitial}]\label{thm:introhamlocinitial}
	Let $\AA\Rightarrow M$ be an almost injective Lie algebroid and let $\pi_\AA$ be a generically non-degenerate Poisson structure on $\AA$. Then $\Hamlocc(\AA, \pi_\AA)$ is a Lie group with Lie algebra $(\Omega^1_{\mathrm{cl},c}(\AA), [\cdot, \cdot]_{\pi_\AA})$.
\end{theorem}
The bracket $[\cdot, \cdot]_{\pi_\AA}$ on $\Gamma(\AA^*)=\Omega^1(\AA)$ is induced by the $\AA$-Poisson structure $\pi_\AA$ (Remark \ref{rk:exactbialgebroids}), and the group $\Hamlocc(\AA, \pi_\AA)$ consists of $\AA$-locally Hamiltonian Poisson automorphisms (see Section \ref{sec:algebraicaspects}).

Note that Theorem \ref{thm:introhamlocinitial} applies in particular to the almost-regular Poisson structures as studied by Androulidakis and Zambon in \cite{androulidakiszambon2017}. 

\subsubsection*{Applications} Concrete applications of our approach are collected in the following list of examples.

\begin{example}[Symplectic foliations] Let $(M, \pi)$ be a Poisson manifold of constant rank with associated symplectic foliation $(\FF, \omega)$. Then $\Fol(M, \pi)$ is a initial Lie subgroup of $\Diff(M)$, with Lie algebra $\Omega^1_{\mathrm{cl},c}(\FF)$ of closed, compactly supported one-forms on $\FF$. Note that this is isomorphic via $\omega^\flat$ to the Lie algebra of foliated vector fields $\fol_c(M, \pi)$. This Lie group structure coincides with the one in \cite{Rybicki2001foliated}. The relevant groupoid is the holonomy groupoid $\Hol(\FF)\toto M$ of the foliation $\FF$.
\end{example}
\begin{example}[Cosymplectic manifolds]
	Let $(M, \pi)$ be a Poisson manifold of $k$-cosymplectic type. Its Poisson diffeomorphism group $\Diff(M, \pi)$ is an embedded Lie subgroup of $\Diff(M)$, with Lie algebra $\mf{X}_c(M, \pi)$ consisting of compactly supported Poisson vector fields.
	When all the symplectic leaves are embedded in $M$, then $\Fol(M, \pi)$ becomes an embedded Lie subgroup of $\Diff(M, \pi)$. 
	
	For this example, the relevant groupoids are the pair Poisson groupoid $(M\times M, \pi\times (-\pi))\toto M$ and the holonomy groupoid of the underlying foliation.
\end{example}
\begin{example}[Log-symplectic manifolds]
	Let $(M, \pi)$ be a normal-crossing log-symplectic manifold. All the groups $\Diff(M, \pi)$, $\Fol(M, \pi)$ and $\Hamlocc(M, \pi)$ are initial Lie subgroups of $\Diff(M)$, whose Lie algebras are isomorphic to $\mf{X}_c(M, \pi)$, $\mf{fol}_c(M, \pi)$ and $\hamlocc(M, \pi)$ respectively. In fact, the space $\fol_c(M, \pi)$ is isomorphic to $\hamlocc(M, \pi)$ and therefore $\Hamlocc(M, \pi)$ coincides with the identity component of $\Fol(M, \pi)$. The relevant groupoids here are the integrations of the \textit{log-tangent bundle} $T_ZM$ and the symplectic groupoid.
\end{example}
\begin{example}[Elliptic symplectic manifolds]
	When $(M, \pi)$ is an elliptic symplectic manifold \cite{Cavalcantigualtieri2017}, then both $\Diff(M, \pi)$, $\Fol(M, \pi)$ and $\Hamlocc(M, \pi)$ are initial Lie subgroups of $\Diff(M)$, with Lie algebras $\mf{X}_c(M, \pi)$, $\fol_c(M, \pi)$ and $\hamlocc(M,\pi)$, respectively. If $(M, \pi)$ has non-zero elliptic residue, then $\Fol(M, \pi)$ is open in $\Diff(M, \pi)$. If $(M, \pi)$ has zero elliptic residue, then $\Hamlocc(M, \pi)$ coincides with the identity component of $\Fol(M, \pi)$. As in the previous example, the relevant groupoids are integrations of the elliptic tangent bundle as well as the symplectic groupoids. 
\end{example}
\begin{example}[Scattering-symplectic manifolds] The class of scattering-symplectic Poisson structures \cite{lanius2020} can also be treated with our approach, with a small caveat: it requires specific linearization results. More precisely, we prove a Lagrangian neighbourhood theorem for scattering-symplectic manifolds in Appendix \ref{app:scattering}.
	\begin{theorem}[Theorem \ref{thm:sclagrangianneighbourhood}]
		Let $(M, \pi)$ be a scattering symplectic manifold, and $i:L\rightarrow M$ a Lagrangian submanifold transverse to the degeneracy locus $Z\subset M$. Then $\pi$ is linearizable around $Z$.
	\end{theorem}
\noindent As a result, we show $\Diff(M, \pi)$ is an initial Lie subgroup of $\Diff(M)$, integrating the Lie algebra $\mf{X}_c(M, \pi)$. Interestingly, the relevant groupoid here is an integration of the log-tangent bundle $T_ZM$, and not an integration of the scattering-tangent bundle.
\end{example}

\subsection*{Acknowledgements} This article, combined with \cite{smilde2021linearization}, is based on the author's master's thesis at Utrecht University. I would like to thank my master's thesis advisor Ioan \Marcut{} for his guidance, many helpful conversations, and for introducing me to central problem of this paper. Further, I want to thank Aldo Witte for many useful discussions, and Rui Loja Fernandes for comments on several drafts. Finally, I thank the referee for extensive feedback on an earlier version of this paper.

%% file: 1Poissondiffeomorphisms.tex
\section{Poisson geometry: Lagrangians and diffeomorphisms}

\subsection{Lagrangian submanifolds} First, we recall the definition of a coisotropic submanifold.
\begin{definition}
	Let $(M, \pi)$ be a Poisson manifold. A submanifold $C\subset M$ is \textit{coisotropic} when $\pi(\alpha_1, \alpha_2)=0$ for all $\alpha_1, \alpha_2\in (TC)^\circ$, where $(TC)^\circ\subset T^*M$ is the annihilator of $TC$. Equivalently, $(TC)^{\bot_\pi}:=\pi^\sharp\left( (TC)^\circ\right) \subset TC$.
\end{definition}

Coisotropic submanifolds play a crucial role in this paper, because of the following proposition, which is folklore.
\begin{proposition}[\cite{weinstein1988}, Corollary 2.2.3]\label{prop:poissonmapcoisotropic}
	Let $(M, \pi_N)$ and $(N, \pi_N)$ be Poisson manifolds. A map $f:M\to N$ is a Poisson map if and only if the graph of $f$ is coisotropic in $(N\times M, \pi_N\times (-\pi_M))$. 
\end{proposition}

\begin{example} The graph of a Poisson diffeomorphism $\varphi:M\to M$ is a coisotropic submanifold in $(M\times M, \pi\times (-\pi))$. There is more to it: the intersection of the graph of $\varphi$ with the symplectic leaves of $(M\times M, \pi\times(-\pi))$ is either empty or a Lagrangian submanifold of the leaf. This hints to a more restrictive class of coisotropic submanifolds, suitable for the study of Poisson diffeomorphisms.
\end{example} 

\begin{definition}[\cite{vaisman1994}, Remark 7.8]\label{def:lagrangian}
	Let $(M, \pi)$ be a Poisson manifold. A submanifold $L\subset M$ is \textit{Lagrangian} when
	\[
	(TL)^{\bot_\pi}:= \pi^\sharp\left( (TL)^\circ\right)=TL\cap \operatorname{im}\pi^\sharp. \qedhere
	\]
\end{definition}
The following lemma provides practical characterizations of Lagrangian submanifolds. 
\begin{lemma}\label{lem:lagrangiancharaterization}
	Let $(M, \pi)$ be a Poisson manifold and $L\subset M$ a submanifold. The following are equivalent:
	\begin{itemize}[noitemsep, topsep=0em]
		\item[\textit{(i).}] The submanifold $L$ is a Lagrangian submanifold of $M$.
		\item[\textit{(ii).}] For every symplectic leaf $(S, \omega_S)$, the intersection $L\cap S$ is a Lagrangian submanifold of $(S, \omega_S)$ satisfying $T(L\cap S)=TL\cap TS$.
		\item[\textit{(iii).}] The submanifold $L$ is coisotropic and $\pi(\alpha_1, \alpha_2)=0$ for all $\alpha_1, \alpha_2\in (\pi^\sharp)^{-1}(TL)$. 
	\end{itemize}
\end{lemma}
\begin{proof}
  \textit{(i). $\Leftrightarrow$ (ii).} Clearly, \textit{(ii).} implies \textit{(i).} For the converse, recall that, because $L$ is coisotropic, the annihilator $(TL)^\circ$ becomes a subalgebroid of the cotangent algebroid $T^*M$, that we call the conormal algebroid. We equip the path-components of $L\cap S$ with the structure of a manifold by identifying them with the leaves of the algebroid $(TL)^\circ\Rightarrow L$. As $(TL)^\circ \Rightarrow$ is a subalgebroid of $T^*M\Rightarrow M$, every leaf $\tilde{L}$ of $(TL)^\circ\Rightarrow L$ is contained in a symplectic leaf $S$. Conversely, if $\gamma:\RR\to L\cap S$ is a path, it must be tangent to $L$ and $S$ simultaneously, i.e. $\dot{\gamma}(t)\in T_{\gamma(t)}L\cap T_{\gamma(t)} S=T_{\gamma(t)}L\cap \Im \pi^\sharp_{\gamma(t)}$ for all $t$. Since $\pi^\sharp((TL)^\circ)=TL\cap \Im \pi^\sharp$, it follows that $\gamma$ is tangent to a leaf $\tilde{L}$ of the conormal algebroid $(TL)^\circ\Rightarrow L$.

	\textit{(i). $\Rightarrow$ (iii).} A Lagrangian submanifold $L$ is clearly coisotropic. Moreover, for $\alpha_1, \alpha_2\in (\pi^\sharp)^{-1}(TL)$ there exists $\tilde{\alpha}_1, \tilde{\alpha}_2\in (TL)^\circ$ with $\pi^\sharp(\alpha_1)=\pi^\sharp(\tilde{\alpha}_1)$ and $\pi^\sharp(\alpha_2)=\pi^\sharp(\tilde{\alpha}_2)$. In particular, $\pi(\alpha_1, \alpha_2)=\pi(\tilde{\alpha}_1, \tilde{\alpha}_2)=0$.
	
	\textit{(iii). $\Rightarrow$ (i).} To prove the converse, we have to find for $\alpha\in (\pi^\sharp)^{-1}(TL)$ an element $\tilde{\alpha}\in (TL)^\circ$ such that $\pi^\sharp(\alpha)=\pi^\sharp(\tilde{\alpha})$. Since $\alpha\vert_{TL\cap\operatorname{im} \pi^\sharp}=0$, we can construct $\tilde{\alpha}$ by extending $\alpha\vert_{\operatorname{im}\pi^\sharp}$ by 0 over $TL$. 
\end{proof}
\begin{remark}
	An analogue for \text{(ii).} in Lemma \ref{lem:lagrangiancharaterization} does not exist for arbitrary coisotropic submanifolds. A counterexample is the following. Consider $\RR^3$, with coordinates $(x, y, z)$, foliated by the planes of constant $z$ with a foliated symplectic form $dx\wedge dy$. The graph $\Gamma$ of $f(x, y)=x^2+y^2$ is coisotropic in $\RR^3$ with respect to this Poisson structure, but the intersection is not clean at the origin. Indeed, the intersection of $\Gamma$ with the $z=0$ plane is just a point, while the intersection of the tangent spaces is two-dimensional. This is also the only point where the Lagrangian condition fails. 
\end{remark}

\subsubsection{Linear Poisson structures and the linearization problem for Lagrangian submanifolds}\label{subsec:linearpoissonstructures}
Let $E \to L$ be a vector bundle, with scalar multiplication $m_t:E\to E$. A Poisson structure $\pi$ on the total space of $E$ is linear when $t m_t^*\pi=\pi$ for all $t\in \RR\setminus \{0\}$. We will see as part of Proposition \ref{prop:coisotropicbisections} that the zero section must be a Lagrangian submanifold of $(E, \pi)$. If $\pi$ is any Poisson structure on the total space of $E$ for which the zero section $L\subset E$ is a Lagrangian, we can call $\pi$ \textit{linearizable} around $L$ when there is a local Poisson diffeomorphism 
\[
(E, \pi)\dashrightarrow (E, \pi_{\mathrm{lin}}), \quad \mbox{with } \pi_{\mathrm{lin}}=\lim_{t\to 0} t m^*_t \pi,
\]
that restricts to the identity on $L$. More generally, if $(M, \pi)$ is any Poisson manifold, and $L\subset (M, \pi)$ a Lagrangian, we can call $(M, \pi)$ linearizable around $L$ if $(NL, \psi^* \pi)$ is linearizable around $L$ for some (hence every) tubular neighbourhood $\psi:NL\to U\subset M$ of $L$. 
\begin{question}
	Let $L\subset (M, \pi)$ be a Lagrangian submanifold. When is $(M,\pi)$ linearizable around $L$?
\end{question}
We have addressed this question in more detail in \cite{smilde2021linearization}, with emphasis on the case that the ambient manifold is a Poisson groupoid, and the Lagrangian is the unit section. As in \cite{smilde2021linearization}, we adopt the following terminology.
\begin{definition}
	A Poisson groupoid $(\GG, \Pi)\toto (M, \pi)$ is \textit{linearizable} when the Poisson structure $\Pi$ is linearizable around the unit space $M\subset (\GG, \Pi)$. 
\end{definition}

\subsection{Poisson Lie algebroids}
A Lie algebroid $\AA\Rightarrow M$ over $M$ consists of a vector bundle $\AA\to M$ together with an anchor $\rho_\AA:\AA\to TM$ and a Lie bracket $[\cdot, \cdot]_\AA$ on $\Gamma(\AA)$ subject to the Leibniz rule: $[v, fw]_\AA=f[v, w]_\AA+\LL_{\rho_\AA(v)}(f) w$. Due to the presence of the bracket, a Lie algebroid behaves much like the tangent bundle of a manifold $M$, and therefore it is often the case that geometry on $TM$ can by `lifted' to a Lie algebroid. In this section we introduce Poisson Lie algebroids as the `lifts' of Poisson structures (see also \cite{klaasse2018}).

Recall that the bracket of a Lie algebroid $\AA\Rightarrow M$ can be extended to a Schouten-Nijenhuis bracket on $\Gamma(\wedge^\bullet\AA)$. 
\begin{definition}
	Let $\AA\Rightarrow M$ be a Lie algebroid. A \textit{Poisson structure on $\AA$}, or an \textit{$\AA$-Poisson structure}, is a section $\pi_\AA\in \Gamma(\wedge^2\AA)$ satisfying $[\pi_\AA, \pi_\AA]_\AA=0$. The pair $(\AA, \pi_\AA)\Rightarrow M$ is referred to as a \textit{Poisson Lie algebroid.}
\end{definition}
\begin{remark}\label{rk:exactbialgebroids}
	A Lie bialgebroid $(\AA, \AA^*)$ consists of a vector bundle $\AA$ with Lie algebroid structures on $\AA$ and $\AA^*$ that are compatible in the following sense:
	\[
	d_{\AA^*}[v, w]_\AA=\left[ d_{\AA^*}v, w\right]_\AA+\left[v, d_{\AA^*}w\right]_\AA
	\]
	for all $v, w\in \Gamma(\AA)$. Given a Poisson Lie algebroid $(\AA, \pi_\AA)\Rightarrow M$, the $\AA$-Poisson structure $\pi_\AA$ anchors $\AA^*$ to $\AA$ via the sharp map determined by $\beta(\pi^\sharp_\AA(\alpha))=\pi_\AA(\alpha, \beta)$ for $\alpha, \beta\in \AA^*$. Also, it induces a bracket on $\Gamma(\AA^*)$ via the formula
	\[
	[\alpha, \beta]_{\pi_\AA}= \LL_{\pi_\AA^\sharp(\alpha)}(\beta)-\LL_{\pi_\AA^\sharp(\beta)}(\alpha)-d_\AA \left(\pi_\AA(\alpha, \beta)\right).
	\] 
Together with the anchor map $\rho_{\AA^*}=\rho_\AA\circ\pi_\AA^\sharp$, the vector bundle $\AA^*$ becomes a Lie algebroid whose differential on $\Gamma(\wedge^\bullet\AA)$ corresponds to $[\pi_\AA, \cdot]$, from which easily follows that the pair $(\AA, \AA^*)$ is a Lie bialgebroid. The data $(\AA, \AA^*, \pi_\AA)$ constitutes a triangular Lie bialgebroid, introduced in \cite{mackenziexu1994}. In the context of Lie bialgebras, the $\AA$-Poisson structure $\pi_\AA$ is called the $r$-matrix.
\end{remark}

Let $(\AA, \pi_\AA)\Rightarrow M$ be a Poisson Lie algebroid. It induces a Poisson structure $\pi$ on $TM$ via the following diagram.
\[
\begin{tikzcd}
	\AA^* \arrow[r, "\pi_\AA^\sharp"] & \AA \arrow[d, "\rho_\AA"]\\
	T^*M \arrow[u, "\rho_\AA^*"] \arrow[r, "\pi^\sharp"] & TM
\end{tikzcd}
\]
Therefore, it is natural to refer to $\pi_\AA$ as the \textit{lift of $\pi$ to $\AA$.}
\begin{definition}
	Let $\AA\Rightarrow M$ be a Lie algebroid. A \textit{symplectic structure} on $\AA$, or an \textit{$\AA$-symplectic structure}, is a closed, non-degenerate two form $\omega_\AA\in \Omega^2(\AA)$. In this case, we call $(\AA, \omega_\AA)$ is a \textit{symplectic Lie algebroid}.
\end{definition}
If $\pi_\AA$ is a non-degenerate Poisson structure on $\AA$, then the 2-from $\omega_\AA\in \Omega^2(\AA)$ defined by
\[
\omega_\AA^\flat=\left(\pi_\AA^\sharp\right)^{-1},
\]
with $\omega_\AA^\flat(v) (w)=\omega_\AA(v, w)$ for $v,w\in \AA$, is closed, and therefore symplectic. We write $\omega_\AA=\pi_\AA^{-1}$. 
\begin{lemma}
Let $\AA\Rightarrow M$ be a Lie algebroid. Then there is a one-to-one correspondence between non-degenerate $\AA$-Poisson structures and symplectic structures on $\AA$, sending $\pi_\AA$ to $\omega_\AA=\pi_\AA^{-1}$. 
\end{lemma}

\subsubsection{Poisson Lie algebroids of cosymplectic type}
In this section we consider a class of $\AA$-Poisson structures of constant rank whose transverse geometry is relatively simple. These will come back in later applications.
\begin{definition}[\cite{osornotorresphd}, Definition 3.2.23]
	Let $\AA\Rightarrow M$ be a Lie algebroid. A \textit{cosymplectic structure} of \textit{type $k$} (or $k$-cosymplectic structure) consists of closed one-forms $\alpha_1, \dots, \alpha_k\in \Omega^1(\AA)$ and a closed 2-form $\omega\in \Omega^2(\AA)$ of constant rank $2n$ such that $\omega^n\wedge\alpha_1\wedge\dots\wedge\alpha_k$ is a nowhere vanishing top-form on $\AA$.
\end{definition}
A cosymplectic structure $(\alpha_1, \dots, \alpha_k, \omega)$ on $\AA\Rightarrow M$ induces an isomorphism
\[
\flat:\AA\to \AA^*, \quad v\mapsto \iota_v\omega+\sum_{i=1}^k\alpha_i(v)\alpha_i,
\]
called the \textit{flat map}. The Reeb sections of the cosymplectic structures are given by $R_i:=\flat^{-1}(\alpha_i)$

The 2-form $\omega$ restricts to an $\FF$-symplectic structure $\omega_\FF=\omega\vert_\FF$ on the $\AA$-foliation $\FF=\cap_{i=1}^k \ker \alpha_i$, and thus determines an $\AA$-Poisson structure $\pi_\AA$ via the diagram
\[
\begin{tikzcd}
	\FF^* & \FF \arrow[l, "\omega_\FF^\flat"'] \arrow[d] \\
	\AA^* \arrow[u] \arrow[r, "\pi^\sharp_\AA"] & \AA.
\end{tikzcd}
\]
\begin{definition}\label{def:cosymplectictype}
	Let $(\AA, \pi_\AA)\Rightarrow M$ be a Poisson algebroid. We call $\pi_\AA$ of \textit{$\AA$-cosymplectic type} if there exists a cosymplectic structure on $\AA$ inducing $\pi_\AA$.
\end{definition}
\begin{remark}
	The Reeb sections $R_i$ of a cosymplectic structure inducing $\pi_\AA$ are `Poisson sections' meaning that  $[\pi_\AA, R_i]=0$.
\end{remark}

\subsection{Algebraic aspects of the Poisson diffeomorphisms group}\label{sec:algebraicaspects}
Throughout this section, we fix a Poisson manifold $(M, \pi)$. A \textit{Poisson diffeomorphism} of $(M, \pi)$ is a diffeomorphism $\varphi\in \Diff(M)$ that is also a Poisson map, i.e. $\varphi_*(\pi)=\pi$. They form the group $\Diff(M, \pi)$ of Poisson diffeomorphisms. A vector field $X\in \mf{X}(M)$ is \textit{Poisson} when $\LL_X(\pi)=0$. They form a subalgebra $\mf{X}(M, \pi)$ of $\mf{X}(M)$ under the Lie bracket. 

The Poisson manifold $(M, \pi)$ has an underlying symplectic foliation, leading to the group of \textit{foliated Poisson diffeomorphisms} $\Fol(M, \pi)$, consisting of the maps that send each symplectic leaf symplectomorphically to itself. Infinitesimally, the \textit{foliated Poisson vector fields}, constituting a subalgebra $\mf{fol}(M, \pi)$ of $\mf{X}(M, \pi)$, are those Poisson vector fields that are tangent to the symplectic foliation. 

The Lie algebra $\mf{X}(M, \pi)$ has two more interesting subalgebras. Let $\Omega^1_{\mathrm{cl}}(M)$ and $\Omega^1_{\mathrm{ex}}(M)$ be the closed and exact one-forms on $M$, respectively. We introduce the Lie algebras of Hamiltonian and locally Hamiltonian vector fields, respectively, as
\[
\mf{ham}(M, \pi)=\pi^\sharp\left( \Omega^1_{\mathrm{ex}}(M)\right), \quad \mbox{and}\quad  \mf{ham}_{\mathrm{loc}}(M, \pi)=\pi^\sharp\left(\Omega^1_{\mathrm{cl}}(M)\right).
\]
A diffeomorphism is \textit{(locally) Hamiltonian} if is the time-1 flow of a (locally) Hamiltonian vector field generated by a time-dependent (closed) exact one-form. These give rise to the group $\operatorname{Ham}_{\mathrm{(loc)}}(M, \pi)$ of (locally) Hamiltonian diffeomorphisms. Clearly, there are inclusions
\[
\begin{tikzcd}
\ham(M, \pi)\arrow[r, hook] &\hamloc(M, \pi)\arrow[r, hook] & \fol(M, \pi)\arrow[r, hook] & \mf{X}(M, \pi),\\
\Ham(M, \pi)\arrow[r, hook] & \Hamloc(M, \pi)\arrow[r, hook] & \Fol(M, \pi) \arrow[r, hook] & \Diff(M, \pi).
\end{tikzcd}
\]
\begin{proposition}
	The groups $\Fol(M, \pi)$, $\Hamloc(M, \pi)$ and $\Ham (M, \pi)$ are normal subgroups of $\Diff(M, \pi)$.  
\end{proposition}
\begin{proof}
	Establishing that $\Ham(M, \pi)$ and $\Hamloc (M, \pi)$ are actually subgroups goes via the usual argument: if $\varphi_t$ and $\psi_t$ are isotopies generated by (locally) Hamiltonian vector fields $\pi^\sharp(\alpha_t)$ and $\pi^\sharp(\beta_t)$, respectively, then $\varphi_t\circ\psi_t$ is generated by \[
	\pi^\sharp(\alpha_t)+(\varphi_t)_*\pi^\sharp(\beta_t)=\pi^\sharp(\alpha_t+(\varphi_t)_*\beta_t),
	\]
	which is again (locally) Hamiltonian. 
	
	The foliated diffeomorphisms clearly form a normal subgroup of $\Diff(M, \pi)$. Normality of $\Ham(M, \pi)$ and $\Hamloc(M, \pi)$ also follows from the identity $\varphi_*(\pi^\sharp(\alpha))=\pi^\sharp(\varphi_*(\alpha))$, which holds for general $\varphi\in \Diff(M, \pi)$ and $\alpha\in \Omega^1(M)$.
\end{proof}
\begin{lemma}\label{lem:poissonisotopy}
	Let $\varphi_t$ be an isotopy on $M$ generated by $X_t$. Then $\varphi_t$ is a path in $\Diff(M, \pi)$ (resp. in $\Fol(M, \pi)$) if and only if $X_t$ is in $\mf{X}(M,\pi)$ (resp. in $\fol(M, \pi)$). 
\end{lemma}
\begin{remark}
	A similar result for the (locally) Hamiltonian groups is hard. For symplectic manifolds, the argument relies on a flux homomorphism, which is not available for general Poisson manifolds. It is one of the interesting open questions regarding the Poisson diffeomorphism group. 
	
	There are other open questions that will not be addressed in this paper. For instance, if a time-dependent vector field $X_t$ is Hamiltonian for all $t$, is it generated by a smooth time-dependent function?
\end{remark}

\begin{remark}
	The infinitesimal algebraic structure of $\mf{X}(M, \pi)$ is encoded in the first Poisson cohomology group. For example,
	\[
	\frac{\mf{X}(M, \pi)}{\ham(M, \pi)}=H^1_\pi(M), \quad \frac{\hamloc(M, \pi)}{\ham(M, \pi)}=\operatorname{im}\left(\pi^\sharp:H^1(M)\to H^1_\pi(M)\right).
	\]
	When $\pi$ has constant rank, and induced by the symplectic foliation $(\mc{F}, \omega)$, then
	\[
	\frac{\fol(M, \pi)}{\ham(M, \pi)}=H^1(\mc{F}), \quad \frac{\fol(M, \pi)}{\hamloc(M, \pi)}=\frac{H^1(\mc{F})}{\operatorname{im}\left( H^1(M)\to H^1(\mc{F})\right)}.
	\]
\end{remark}
\begin{remark}
	There are many more interesting subgroups of $\Diff(M,\pi)$. For instance, it is very natural to study \textit{leafwise Hamiltonian diffeomorphisms}, which are foliated Poisson diffeomorphisms that restrict to a Hamiltonian diffeomorphism on each leaf. All these different subgroups illustrate that the algebraic structure of the Poisson diffeomorphism group is more intricate than the symplectomorphism group of a symplectic manifold: the latter is determined by the Hamiltonian diffeomorphism and (a quotient of) $H^1(M)$ (via a flux homomorphism, see \cite{mcduffsalamon2007}, Chapter 10).
\end{remark}

Many of the concepts introduced in this section can be defined on an arbitrary Poisson Lie algebroid $(\AA, \pi_\AA)\Rightarrow M$, as we will now briefly discuss. A \textit{Poisson section} is a section $v\in \Gamma(\AA)$ such that $\LL_v(\pi_\AA)=[v, \pi_\AA]=0$, defining a subalgebra $\Gamma(\AA, \pi_\AA)$ of $\Gamma(\AA)$. The closed $\AA$-forms, $\Omega^1_{\mathrm{cl}}(\AA)$ form a Lie subalgebra of $\Omega^1(\AA)$ under the bracket on $\AA^*$ (Remark \ref{rk:exactbialgebroids}). Therefore, we obtain a subalgebra of $\Gamma(\AA, \pi_\AA)$, called the $\AA$-locally Hamiltonians, by
\[
\hamloc(\AA, \pi_\AA)=\pi_\AA^\sharp\left(\Omega^1_{\mathrm{cl}}(\AA) \right).
\]
An automorphism of $\AA$ is $\AA$-locally Hamiltonian when it is the time-1 flow of a $\AA$-locally Hamiltonian vector field generated by a time-dependent closed $\AA$-form, forming the group $\Hamloc(\AA, \pi_\AA)$ regarded as a subgroup of the group $\InnAut(\AA)$ of inner automorphisms of $\AA$.

\subsection{Poisson structure of divisor type}\label{sec:divisortype}
To provide context to several central examples considered in this paper, we give a brief summary of the Poisson structures of divisor type developed by Klaasse in \cite{klaasse2018}. 

A (real) \textit{divisor} on a smooth manifold $M$ is a pair $(L, \sigma)$ consisting of a (real) line bundle $L\to M$ and a section $\sigma\in \Gamma(L)$ whose zero locus $Z_\sigma$ is nowhere dense. 

The section $\sigma$ of a divisor $(L, \sigma)$ defines an evaluation map $\sigma:\Gamma(L^*)\to C^\infty(M)$, and the ideal of $C^\infty(M)$ associated to $(L, \sigma)$ is its image $I_\sigma=\sigma(\Gamma(L^*))$. Note that the \textit{vanishing set $Z_{I_\sigma}$} of $I_\sigma$ is exactly $Z_\sigma$. The vanishing ideal $I_{Z_\sigma}$ of $Z_\sigma$ is generally larger than $I_\sigma$. An ideal $I\subset C^\infty(M)$ that comes from a divisor is a \textit{divisor ideal.}

Every divisor ideal is locally principal, as it is generated by $\langle\alpha, \sigma\rangle$, where $\alpha$ is a local trivialization of $L^*$.

\begin{definition}[\cite{klaasse2018}, Definition 3.6]
	A Lie algebroid $\AA\Rightarrow M$ is \textit{of divisor type} when 
	\[
	\operatorname{div}(\AA)=\left(\det(\AA^*)\otimes \det(TM), \det(\rho_\AA)\right)
	\] 
	is a divisor, or, equivalently, when the isomorphism locus of $\rho_\AA$ is dense in $M$. The divisor ideal associated to $\operatorname{div}(\AA)$ is denoted $I_\AA$. 
\end{definition}
\begin{remark}
	A Lie algebroid $\AA\Rightarrow M$ is of divisor type if and only if the isomorphism locus of the anchor $\rho_\AA$, which is the set on which $\rho_\AA$ is an isomorphism, is dense in $M$. 
\end{remark}

Let $I\subset C^\infty(M)$ be an ideal and $\AA\Rightarrow M$ a Lie algebroid. Following \cite{klaasse2018}, we set
\[
\Gamma(\AA)_I=\left\{ v\in \Gamma(\AA): \LL_v(I)\subset I\right \},
\]
the set of sections of $\AA$ that preserve $I$. This is a subalgebra of $\Gamma(\AA)$ (\cite{klaasse2018}, Lemma 3.41). 

When $\Gamma(\AA)_I$ is locally finitely generated and projective, the Serre-Swan theorem gives a vector bundle $\AA_I\to M$ such that $\Gamma(\AA_I)=\Gamma(\AA)_I$. The inclusion of sections determines a bundle map $\rho_{\AA_I}^\AA:\AA_I\to \AA$. Together with the Lie bracket on $\Gamma(\AA_I)$ inherited from $\Gamma(\AA)$, and with anchor $\rho_{\AA_I}=\rho_{\AA_I}^\AA\circ \rho_{\AA}$, it becomes a Lie algebroid anchored to $\AA$.
\begin{definition}[\cite{klaasse2018}, Definition 3.42]
	Let $\AA\Rightarrow M$ be a Lie algebroid and $I\subset C^\infty(M)$ an ideal. If it exists, the Lie algebroid $\AA_I\Rightarrow M$ is called the \textit{(primary) ideal Lie algebroid} associated to $(\AA, I)$. 
\end{definition}
This is often applied in the case that $\AA=TM$, for which the following terminology is introduced in \cite{klaasse2018}. 
\begin{definition}[\cite{klaasse2018}, Definition 3.43]\label{def:projectivedivisor}
	A divisor ideal $I$ is \textit{projective} when the ideal Lie algebroid $T_IM:=TM_I$ exists. It is \textit{standard} when it is projective and $I_{T_IM}=I$. 
\end{definition}
By the Serre-Swan theorem, a divisor ideal $I$ is projective when $\mf{X}(M)_I$ is projective as a $C^\infty(M)$-module \textit{and} when it is locally finitely generated. 

\begin{definition}[\cite{klaasse2018}, Definition 4.3]
 A Poisson structure $(M^{2n}, \pi)$ is of divisor type when $(\wedge^{2n}TM, \wedge^n\pi)$ is a divisor. The associated divisor ideal is denoted by $I_\pi$. 
\end{definition}

One of the main results in \cite{klaasse2018} (Theorem 4.35) is that whenever $I_\pi$ is projective and $T^*_{I_\pi}M$ locally admits bases of closed sections, then $\pi$ lifts to a Poisson structure on $T_{I_\pi}M$. If in addition $I_\pi$ is standard, the lift is non-degenerate. This fact is used in the examples below.

Before we discuss some examples, we want to highlight the following proposition, which explains the significance of the Lie algebroid $T_{I_\pi} M$ (if it exists) associated to a Poisson manifold of divisor type $(M,\pi)$ from the perspective of Poisson diffeomorphisms: because $\Gamma(T_{I_\pi}M)$ contains all the Poisson vector fields, an integration of $T_{I_\pi}M$ can be a good starting point to describe the Poisson diffeomorphism group as a Lie group. This observation will be exploited in Section \ref{sec:divisortypdiffeos}.
\begin{proposition}\label{prop:poissonprojectivedivisor}
Let $(M, \pi)$ be a Poisson manifold of divisor type whose divisor ideal $I_\pi$ is projective. Then any Poisson vector field lifts to a section of $T_{I_\pi}M$. 	
\end{proposition}
\begin{proof}
	Let $f\in I_\pi$. Then there exists a top-form $\omega\in \Omega^{2n}(M)$ such that $f=\langle \omega, \wedge^n \pi\rangle$. If $X\in \mf{X}(M)$ is a Poisson vector field, then
	\begin{align*}
		\LL_X(f)&=\langle \LL_X(\omega), \wedge^{n}\pi\rangle+\langle \omega, \LL_X(\wedge^n\pi)\rangle =\langle \LL_X(\omega), \wedge^n \pi\rangle\in I_\pi.
	\end{align*}
It follows that $X$ is a section of $T_{I_\pi} M$.
\end{proof}

\begin{example}[Log-symplectic manifolds]\label{ex:logsymplectic}
A \textit{(normal-crossing) log-manifold}\footnote{Also called a \textit{$c$-manifold} in \cite{mirandascott2020}.} is a pair $(M, Z)$ of a manifold $M$ and a (closed) hypersurface $Z\subset M$ that is allowed to have self-intersections. More precisely, around each point $z\in Z$ there are coordinates $(x_1, \dots, x_k, x_{k+1}, \dots, x_n)$ for which $Z=\{\prod_{i=1}^k x_i=0\}$. Note that $k$ need not be fixed, and is allowed to vary between different points in $Z$. We call such coordinates \textit{adapted to $Z$}. When $Z$ is an embedded hypersurface, we refer to the pair $(M, Z)$ as a \textit{smooth} log-manifold. 

The vanishing ideal $I_Z$ is a standard divisor ideal, and a divisor $(L, \sigma)$ with $I_\sigma=I_Z$ is called a \textit{normal-crossing} log divisor. The corresponding Lie algebroid $T_Z M:=T_{I_Z}M$ is called the \textit{log-tangent bundle}. The space of sections of $T_ZM$, called log-vector fields, is denoted by $\mf{X}(M, Z)$ and is locally generated by the vector fields $x_1\del_{x_1}, \dots, x_k\del_{x_k}, \del_{x_{k+1}}, \dots, \del_{x_n}$ in an adapted coordinate chart.

\begin{definition}
	A Poisson structure $\pi$ on $M^{2n}$ is \textit{log-symplectic} when $(\wedge^{2n}TM, \wedge^n \pi)$ is a normal-crossing log divisor. 
\end{definition}
By Theorem 4.35 in \cite{klaasse2018}, a log-symplectic structure lifts to a symplectic structure on the log-tangent bundle $T_ZM$.
\end{example}
\begin{example}[Elliptic symplectic manifolds]\label{ex:ellipticpoisson} A (real) divisor $|D|=(L, \sigma)$ is \textit{elliptic} when $D:=Z_\sigma$ is a codimension 2 submanifold and the normal Hessian $\operatorname{Hess}(\sigma)\in \Gamma(\operatorname{Sym}^2N^*D \otimes L)$ is positive-definite \cite{Cavalcantigualtieri2017}. Its associated ideal is usually denoted by $I_{|D|}$. An ideal $I\subset C^\infty(M)$ is \textit{elliptic} when it is the ideal of some elliptic divisor. 
	
An elliptic ideal $I_{|D|}$ is standard, and the associated Lie algebroid $T_{|D|}M\Rightarrow M$ is called the \textit{elliptic tangent bundle.} For any elliptic divisor, there are coordinates $(x_1, \dots, x_n)$ such that $I_{|D|}=\langle x_1^2+x_2^2\rangle$. Sections of $T_{|D|}M$ are locally generated by $r\del_r, \del_\theta, \del_{x_3}, \dots, \del_{x_n}$. 
\begin{definition}
	A Poisson structure $\pi$ on $M^{2n}$ is \textit{elliptic} when $(\wedge^{2n}TM, \wedge^n\pi)$ is an elliptic divisor. 
\end{definition}
Elliptic symplectic structures appear naturally as the Poisson structure associated to a stable generalized complex structure \cite{Cavalcantigualtieri2017}. By Theorem 4.35 in \cite{klaasse2018}, an elliptic symplectic structure $\pi$ on $M$ lifts to a symplectic structure on $T_{|D|}M$, with $|D|=(\wedge^{2n}TM, \wedge^n\pi)$ (cf. \cite{Cavalcantigualtieri2017}, Lemma 3.4).
\end{example}

\subsection{Scattering-symplectic manifolds}
Let $\AA\Rightarrow M$ be a Lie algebroid. A subalgebroid of $\AA$ is a Lie algebroid $\BB\Rightarrow Z$ together with an injective Lie algebroid homomorphism $(\iota, i):\BB\to \AA$ that makes $i:Z\hookrightarrow M$ into a submanifold of $M$. Usually, we identify a subalgebroid $\BB\Rightarrow Z$ with its image in $\AA$.

Given a subalgebroid $\BB\Rightarrow Z$ of $\AA\Rightarrow M$, we consider the module of sections of $\AA$ that are "tangent" to $\BB$:
\[
\Gamma(\AA, \BB)=\left \{ v\in \Gamma(\AA): v\vert_Z\in \Gamma(\BB)\right \}.
\]
If $Z\subset M$ is a closed hypersurface, this is a locally finitely generated subalgebra of $\Gamma(\AA)$ and therefore by the Serre-Swan theorem corresponds to a vector bundle $[\AA;\BB]\rightarrow M$ that inherits a Lie algebroid structure from $\AA$ \cite{gualtierili2014,lanius2020}. The Lie algebroid $[\AA;\BB]\Rightarrow M$ is the \textit{rescaling} of $\AA$ by $\BB$.

\begin{example}\label{ex:logrescaling}
	Let $(M, Z)$ be a smooth log-manifold. Then $TZ\Rightarrow Z$ is a subalgebroid of $TM$, and the rescaling $[TM;TZ]$ is exactly the log-tangent bundle $T_ZM$ from Example \ref{ex:logsymplectic}.
\end{example}
\begin{example}[Zero-rescalings]\label{ex:0rescaling}
	The zero algebroid $0_Z\Rightarrow Z$ is a subalgebroid of any Lie algebroid $\AA\Rightarrow M$. Then $[\AA; 0_Z]$ is called the \textit{zero-rescaling} of $\AA$ over $Z$. In the case that $\AA=TM$, the resulting algebroid is called the \textit{0-tangent bundle} $^0T_ZM$. It's sections correspond to vector fields on $M$ that vanish over $Z$.
\end{example}
	
\begin{example}[Scattering manifolds]\label{ex:scatteringsymplectic}
	More interesting is the zero-rescaling of $T_ZM$ over $Z$. The resulting Lie algebroid ${}^{sc}T_ZM=[T_ZM;0_Z]$ is the \textit{scattering-tangent} bundle. A Poisson structure $\pi$ in $M$ is scattering-symplectic when it lifts to a symplectic structure $\omega=\pi^{-1}$ on ${}^{sc}T_ZM$. We refer to $(M, Z, \omega)$ as a scattering-symplectic manifold.
	
	Scattering-symplectic structures are ubiquitous. For example, every even dimensional sphere admists a scattering-symplectic structure, whose hypersurface is the equator \cite{lanius2020}.
\end{example}
The following proposition is a consequence of the Serre-Swan theorem. 
\begin{proposition}\label{prop:rescalingmorphisms}
	Let $\AA\Rightarrow M$ be a Lie algebroid and $\BB\Rightarrow Z$ a subalgebroid over a hypersurface $Z\subset M$. Let $\varphi:\AA\to \AA$ be a bundle map that restricts to a bundle map $\varphi\vert_\BB=\psi:\BB\to \BB$. Then $\varphi$ lifts to a (unique) map $[\varphi; \psi]:[\AA; \BB]\to [\AA;\BB]$. If $\varphi$ and $\psi$ are Lie algebroid morphisms, then so is $[\varphi;\psi]$. 
\end{proposition}
As a consequence, every morphism of the log-tangent bundle $T_Zf:T_ZM\to T_ZM$ induces a morphisms of the scattering-tangent bundle ${}^{sc}T_Z f:{}^{sc}T_ZM\to {}^{sc}T_ZM$.

%% file: 3LiegroupsofPoissondiffeomorphisms.tex
\section{Coisotropic bisections}
In this section, we establish general properties of coisotropic bisections of Poisson groupoids and their connection to Poisson diffeomorphisms. After that, we show that they form a Lie group whenever the Poisson groupoid is linearizable. 

\subsection{Coisotropic bisections of a Poisson groupoid}
Let $\GG\toto M$ be a groupoid. A bisection $\sigma\in \Bis(\GG)$ is a section of the source map $s:\GG\to M$ with the property that $l_\sigma:=t\circ \sigma:M\to M$ is a diffeomorphism. Any bisection corresponds to and is determined by a left translation map 
\[
L_\sigma:\GG\to \GG, \quad g\mapsto m(\sigma(t(g)), g). 
\]
If $X\in \Gamma(\AA)$ is a section of the Lie algebroid of $\AA$, then the flow $\varphi_t$ of its right-invariant vector field $\rar{X}$ gives rise to a family of bisections $\sigma_t=\varphi_t\vert_M$, and in fact, it holds that $\varphi_t=L_{\sigma_t}$.

Bisections of a Lie groupoid form a group, and can be multiplied: if $\sigma, \tau\in \Bis(\GG)$, we define $\sigma \cdot \tau$ as
\[
\sigma\cdot \tau(x)=m( \sigma( t\circ \tau(x)), \tau(x)), \quad \mbox{ for $x\in M$,}
\] 
or, equivalently, as $L_{\sigma\cdot \tau}=L_{\sigma} \circ L_{\tau}$. Several authors already described $\Bis(\GG)$ as a Lie group, e.g. in \cite{Rybicki2002, SchmedingWockel2015}.

A Poisson groupoid $(\GG, \Pi)\toto (M, \pi)$ is a Lie groupoid $\GG\toto M$ with a Poisson structure $\Pi$ on the space of arrows that is multiplicative: the graph of the multiplication map is a coisotropic submanifold of $(\GG\times \GG\times \GG, (-\Pi)\times \Pi\times \Pi)$ \cite{weinstein1988}. For each multiplicative Poisson structure $\Pi$ on a groupoid $\GG\toto M$, there is a unique Poisson structure $\pi$ on $M$ for which the target map $t:(\GG, \Pi)\toto (M, \pi)$ is a Poisson map. This allows us to talk about Poisson groupoids \textit{over} a fixed Poisson manifold $(M, \pi)$. 
\begin{definition}
	Let $(\GG, \Pi)\toto (M, \pi)$ be a Poisson groupoid. A bisection $\sigma\in \Bis(\GG)$ is \textit{coisotropic} when $\operatorname{im}\sigma$ is a coisotropic submanifold of $(\GG, \Pi)$. The collection of coisotropic bisections is denoted by $\Bis(\GG, \Pi)$. 
\end{definition}
We establish some properties of coisotropic bisections. 
\begin{proposition}\label{prop:coisotropicbisections}
	Let $(\GG, \Pi)\toto (M, \pi)$ be a Poisson groupoid.
	\begin{itemize}[noitemsep, topsep=0em]
		\item[(i).] The coisotropic bisections form a subgroup $\Bis(\GG, \Pi)$ of $\Bis(\GG)$. 
		\item[(ii).] A bisection $\sigma\in \Bis(\GG)$ is coisotropic if and only if $L_\sigma$ is a Poisson map.
		\item[(iii).] For any $\sigma\in \Bis(\GG, \Pi)$, the induced diffeomorphism on the base $l_\sigma:M\to M$ is a Poisson map.
		\item[(iv).] The image of a coisotropic bisection is a Lagrangian submanifold.
	\end{itemize} 	
\end{proposition}
\begin{remark}
	Although every coisotropic bisection is Lagrangian for any Poisson groupoid, we reserve the term "Lagrangian bisection" (a bisection with Lagrangian image) for symplectic groupoids only. 
\end{remark}
\begin{proof}
	The first assertion follows from the second, while \textit{(ii)} and \textit{(iii)} can be proved by means of the coisotropic calculus developed in \cite{weinstein1988}. Let us briefly sketch how. The central idea in \cite{weinstein1988} is to relax the notion of a Poisson map from $(M, \pi_M)$ to $(N, \pi_N)$ to that of a \textit{Poisson relation} $r:(M, \pi_M)\dashrightarrow (N, \pi_N)$, which is a coisotropic submanifold of $(N\times M, \pi_N\times (-\pi_M))$. By Proposition \ref{prop:poissonmapcoisotropic}, every Poisson map defines a Poisson relation via its graph. Weinstein proved in \cite{weinstein1988} that, under a "very clean" condition (loc. cit., Definition 1.3.7), Poisson relations compose to Poisson relations (loc. cit., Theorem 1.3.9).  
	
	To prove \textit{(ii)} and \textit{(iii)}, we only have to write the maps as a composition of coisotropic relations, and check that they satisfy the "very clean" condition. For instance, if $\sigma$ is a coisotropic bisection, then its image can be regarded as a relation $\sigma:\{*\}\dashrightarrow (\GG, \Pi)$. We can write $L_\sigma$ as a composition of the Poisson relations
	\[
	\sigma\times \id =\{ (k, h, g)\in \GG^3 : g=h, k\in \Im \sigma\}:\GG\dashrightarrow \GG^2
	\]
	and the multiplication map $m:\GG^2\dashrightarrow \GG$. It is slightly tedious but straightforward to show that this composition is "very clean", so that it composes to a Poisson relation, which is exactly the map $L_\sigma$. Conversely, if $L_\sigma$ is a Poisson map, then $\Im \sigma= L_\sigma(M)$ is coisotropic because $M$ is. Statement \textit{(iii)} can be proved in a similar fashion, using that the target map $t:(\GG, \Pi)\to (M, \pi)$ is Poisson.
	
	As for \textit{(iv)}, we only have to show that the unit section $M\subset \GG$ is a Lagrangian, because of \textit{(ii)}. To do so, we apply Lemma \ref{lem:lagrangiancharaterization}. For $\alpha_1, \alpha_2\in (\Pi^\sharp)^{-1}(TM)$ we can write\[
	\alpha_1=\beta_1+s^*\gamma_1, \quad \alpha_2=\beta_2+t^*\gamma_2,
	\]
	for $\beta_1, \beta_2\in (TM)^\circ$ and $\gamma_1, \gamma_2\in T^*M$. It follows that
	\[
	\Pi(\alpha_1, \alpha_2)=\Pi(s^*\gamma_1, t^*\gamma_2)+\Pi(\alpha_1, \beta_2)+\Pi(\beta_1, \alpha_2)-\Pi(\beta_1, \beta_2). 
	\]
	The first term vanishes because $(\ker Ts)^\circ$ and $(\ker Tt)^\circ$ are Poisson orthogonal for Poisson groupoids (\cite{weinstein1988}, Proposition 4.2.6), the second and third therm are zero because $\beta_i\in (TM)^\circ$ and $\Pi^\sharp(\alpha_j)\in TM$, and the last term vanishes because $M$ is coisotropic in $(\GG, \Pi)$. Thus, $M$ is Lagrangian.
\end{proof}

\begin{proposition}[Curves of coisotropic bisections]\label{prop:coisotropicbisectionsliealgebra}
	Let $(\GG, \Pi)\toto M$ be a Poisson groupoid with Lie bialgebroid $(\AA, \AA^*)$. The following Lie algebras are isomorphic.
	\begin{itemize}[noitemsep, topsep=0em]
		\item[(i).] The Lie algebra $\mf{X}^R(\GG, \Pi)$ of right-invariant Poisson vector fields, equipped with the Lie bracket.
		\item[(ii).] The Lie algebra $\Gamma(\AA, d_{\AA^*})$ of $\AA^*$-closed sections of $\AA$, equipped with the bracket on $\Gamma(\AA)$.  
	\end{itemize} 
	The isomorphism is given by restricting $\rar{v}\in \mf{X}^R(\GG, \Pi)$ along $M$. 
	Moreover, if an isotopy $\sigma_t$ of bisections on $\GG$ is generated by $v_t\in \Gamma(\AA)$, then $\sigma_t\in \Bis(\GG, \Pi)$ for all $t$ if and only if $v_t\in \Gamma(\AA, d_{\AA^*})$. 
\end{proposition}
\begin{proof}
	The differential $d_{\AA^*}$ is determined by $\rar{d_{\AA^*} v}=[\Pi, \rar{v}]$ for $v\in \Gamma(\AA)$ (\cite{mackenzie2005}, Corollary 11.4.9), which proves that the correspondence is an isomorphism. The second statement now follows directly from combining Proposition \ref{prop:coisotropicbisections}, \textit{(ii)} with Lemma \ref{lem:poissonisotopy}.
\end{proof}

\begin{example}[Linear Poisson structures] It is well-known that there is a duality between linear Poisson structures and Lie algebroids (\cite{mackenzie2005}, Chapter 9). Therefore, we always think of a vector bundle $\AA^*\to M$ with a linear Poisson structure $\pi_{\mathrm{lin}}$ on its total space as the dual of a Lie algebroid $\AA\Rightarrow M$. Since $(\AA^*, \pi_{\mathrm{lin}})$ can be interpreted as a Poisson groupoid, a section $\alpha:M\to \AA^*$ with coisotropic image is always a Lagrangian submanifold by Proposition \ref{prop:coisotropicbisections}. The coisotropic bisections of $(\AA^*, \pi_{\mathrm{lin}})$ form a vector space by Proposition \ref{prop:linearcoisotropic} below, and will play a crucial role in the proof of Theorem \ref{thm:coisotropicbisectionsliegroup}.
\end{example}
\begin{proposition}[\cite{smilde2021linearization}, Corollary 2.10]\label{prop:linearcoisotropic}
	Let $\AA\Rightarrow M$ be a Lie algebroid. The image of a section $\alpha\in \Gamma(\AA^*)$ is coisotropic in $(\AA^*, \pi_{\mathrm{lin}})$ if and only if $d_\AA \alpha=0$, in which case it is a Lagrangian submanifold.
\end{proposition}

\subsection{The Lie group of (coisotropic) bisections}\label{sec:bisectionsliegroup}
The convenient setting of Kriegl and Michor \cite{krieglmichor1997} is a framework for infinite-dimensional manifolds that is based on smooth curves. At it's core lies Boman's theorem, which states that a map $f:M\to N$  between finite dimensional manifolds $M$ and $N$ is smooth if for every smooth curve $c$ in $M$, the curve $f\circ c$ is smooth in $N$. Likewise, smooth maps between convenient manifolds are characterized by the same property, making smoothness easy to check. 

A (convenient) Lie group $G$ is a convenient manifold together with a group structure for which the multiplication $m:G\times G\to G$ and inversion $\iota:G\to G$ are smooth. The tangent space $\mf{g}=T_eG$ at the identity inherits a Lie bracket from $G$. 

In the following, we set $C^\infty((\RR, 0), (G, e))$ to be the space of smooth curves $\gamma:\RR\to G$ with $\gamma(0)=e$. The \textit{(right-)logarithmic derivative} $\delta^r:C^\infty((\RR, 0), (G, e))\to C^\infty(\RR, \mf{g})$ is defined on a smooth curve $\gamma$ by
\[
(\delta^r\gamma)(t)=\frac{d}{d\epsilon}\big\vert_{\epsilon=0} \gamma(t+\epsilon)\gamma(t)^{-1}.
\]
A Lie group is \textit{regular} when the inverse of $\delta^r$,
\[
\operatorname{Evol}^r_G:C^\infty(\RR, \mf{g})\to C^\infty((\RR, 0), (G, e))
\]
exists and is bijective, and the time-1 evolution map
\[
\operatorname{evol}^r_G:C^\infty(\RR, \mf{g})\to G, \quad \operatorname{evol}^r_G(\gamma)=\operatorname{Evol}^r_G(\gamma)(1)
\]
is smooth. When restricted to constant curves in $\mf{g}$, the evolution map coincides with the usual exponential map for Lie groups. All the Lie groups in this paper are regular. In fact, to the author's knowledge, up until now all known convenient Lie groups are regular. 

Let $G$ be a Lie group. An \textit{initial} Lie subgroup is a Lie group $H$ together with an injective smooth group homomorphism $i:H\to G$ which is an initial map, meaning that a curve $\gamma:\RR\to H$ is smooth if and only if $i\circ \gamma:\RR\to G$ is. Note that this implies that the tangent map $T_hi:T_hH\to T_hG$ is injective. We do not consider immersed submanifolds in the context of infinite dimensional manifolds.

An initial Lie subgroup $i:H\to G$ is \textit{embedded} when $H$ is an embedded submanifold of $G$, that is, $i:H\to G$ is a topological embedding and $G$ admits charts $(\mc{V}, \Phi)$ modelled on a convenient vector space $V$ covering $H$ for which $\Phi(H\cap \mc{V})=W\cap \Phi(\mc{V})$, where $W\subset V$ is a closed subspace. 

The following lemma will be of use later. 

\begin{lemma}[\cite{krieglmichor1997}, 38.7]\label{lem:subgroupregular}
	Let $G$ be a regular Lie group and $i:K\hookrightarrow G$ an initial subgroup with the following property: there exists an open neighbourhood $U$ of $e$ in $G$ and a smooth mapping $p:U\to E$ into a convenient vector space $E$ such that $p^{-1}(0)=K\cap U$ and $p$ is constant on the left cosets $Kg\cap U$. Then $K$ is regular. 
\end{lemma}

\subsubsection{The Lie group of bisections} Interestingly, the group of bisections of a Lie groupoid can always be given the structure of a Lie group. 
\begin{theorem}[\cite{Rybicki2002}, Theorem 2.2]\label{thm:bisectionsliegroup}
	Let $\GG\toto M$ be a (not necessarily Hausdorff) Lie groupoid with Lie algebroid $\AA\Rightarrow M$ over a Hausdorff manifold $M$. Then $\Bis(\GG)$ is a regular (Hausdorff) Lie group integrating the Lie algebra $\Gamma_c(\AA)$ of compactly supported sections of $\AA$. 
\end{theorem}
\begin{remark}
A more recent paper on the group of bisections of a Lie groupoid is the work by Schmeding and Wockel \cite{SchmedingWockel2015}. While the current paper is set up in in the convenient setting, they are using the so-called Bastiani calculus to describe their infinite-dimensional manifolds. An fairly recent overview of Lie theory in infinite dimensions, using the Bastiani calculus, can be found in \cite{neeb2015}. Although the Bastiani calculus seems to be more popular in the literature (to the author's awareness), we chose to work in the convenient setting for two reasons: first, it requires a less technical setup, and second, the paper by Rybicki \cite{Rybicki2002}, using the convenient setting, proves Theorem \ref{thm:bisectionsliegroup} also for non-Hausdorff Lie groupoids.

For the spaces considered in this paper, the two settings aren't too far off: they even agree up to the level of Fr\'echet spaces. Thus, one can expect the results in this paper to be valid in the MB-setting. 
\end{remark}
To proceed to the group of coisotropic bisections, a description of the smooth structure on $\Bis(\GG)$ is necessary. In this paper, the arrow space of a groupoid is assumed to be finite-dimensional.

\subsubsection*{Charts} A chart of $\Bis(\GG)$ around a bisection $\sigma$ can be obtained from any tubular neighbourhood as follows. Identify $N\operatorname{im}\sigma \cong \sigma^*\ker Ts=:\AA_\sigma$. Let $\varphi:\AA_\sigma \supset U \to V\subset \GG$ be a tubular neighbourhood around $\operatorname{im} \sigma$. Define the $C^1$-opens
\begin{align*}
	\mc{U}&=\left\{ v\in \Gamma_c(\AA_\sigma): \substack{ \mbox{ $\operatorname{im} v \subset U$} \\ 
		\mbox{ $s\circ \varphi \circ v$ and $t\circ \varphi\circ v$ are diffeo's}} \right\},\\
	\mc{V}&=\left\{ \tau \in \Bis_c(\GG): \substack{ \mbox{$\operatorname{im}\tau \subset V$} \\
		\mbox{ $p_{\AA_\sigma}\circ \varphi^{-1}\circ \tau$ is a diffeo} }\right\}, 
\end{align*}
and the map $\Phi:\VV\to \mc{U}$ by
\[
\Phi(\tau)=\varphi^{-1}\circ \tau \circ \left( p_{\AA_\sigma}\circ \varphi^{-1}\circ \tau\right)^{-1}\in \mc{U}.
\]
Then $(\mc{V}, \Phi)$ is a chart modelled on $\Gamma_c(\AA_\sigma)$ around the identity. The crux is that any tubular neighbourhood around $\operatorname{im} \sigma$ defines a chart of $\Bis(\GG)$ around $\sigma$ this way. 

\subsubsection*{Topological and smooth structure}\label{subsubsec:topologicalandsmoothstructure} Let $E\to M$ be a vector bundle. The space of smooth sections $\Gamma(E)$, together with the topology of uniform convergence on compact subsets of each derivative separately, is a topological vector space (which is Fréchet when $M$ is second countable) (\cite{krieglmichor1997}, 30.3). A curve $\gamma:\RR\to \Gamma(E)$ is smooth if and only if the associated map $\gamma^\wedge:\RR\times M\to E$ is smooth (\cite{krieglmichor1997}, 27.17). 

Consider a compact set $K\subset M$, and denote by $\Gamma_K(E)$ the space of sections of $E$ with support in $K$. This is a closed subspace of $\Gamma(E)$. The space of compactly supported sections $\Gamma_c(E)$ is endowed with the direct limit locally convex topology inherited from $\Gamma_K(E)$, ranging over compact subsets $K$ of $M$ (\cite{krieglmichor1997}, 30.4).

A curve $\gamma:\RR\to\Gamma_c(E)$ is smooth if and only if the associated map $\gamma^\wedge:\RR\times M\to E$ is smooth and is \textit{timely proper} in the following sense: for every bounded interval $[a,b]\subset \RR$ there exists a compact $K\subset M$ for which $\gamma(t)$ is constant on $M\setminus K$ for all $t\in [a,b]$ (\cite{krieglmichor1997}, 42.5).

As the group of bisections of a Lie groupoid $\GG\toto M$ is modelled on $\Gamma_c(\AA)$, we have the following characterization of smooth curves in $\Bis(\GG)$.
\begin{lemma}[\cite{Rybicki2001}, Lemma 3.3]
	A curve $\gamma:\RR\to \Bis(\GG)$ is smooth if and only its associated map $\gamma^\wedge:\RR\times M \to \GG$ is smooth and $\gamma$ is timely proper: for every bounded interval $[a,b]\subset \RR$ there exists a compact $K\subset M$ such that $\gamma(t)$ is constant on $M\setminus K$ for all $t\in [a,b]$. 
\end{lemma}

\begin{remark}
	Convenient manifolds are assumed to be \textit{smoothly Hausdorff}, which means that the space of smooth functions separates points in the manifold. The group of bisections of a Lie groupoid is smoothly Hausdorff, because it is topologically Hausdorff and is modeled on a convenient vector space that admits partitions of unity (\cite{krieglmichor1997}, 16.10).  
\end{remark}

\subsubsection{The Lie group of coisotropic bisections}
Under the assumption of linearizability of a Poisson groupoid $(\GG, \Pi)$ around $M$, the coisotropic bisections form a Lie group.
\begin{theorem}\label{thm:coisotropicbisectionsliegroup}
	Let $(\GG, \Pi)\toto (M, \pi)$ be a Poisson groupoid with bialgebroid $(\AA, \AA^*)$. Assume that $(\GG, \Pi)$ is linearizable around $M$. Then $\Bis(\GG, \Pi)$ is a regular embedded Lie subgroup of $\Bis(\GG)$, whose Lie algebra corresponds to $\Gamma_c(\AA, d_{\AA^*})=\{ v\in \Gamma_c(\AA): d_{\AA^*} v=0\}$, the space of closed compactly supported one-forms of $\AA^*$. 
\end{theorem}

\begin{proof}
	We will construct charts of $\Bis(\GG)$ adapted to $\Bis(\GG, \Pi)$. Choose a linearization $\varphi:(\AA, \pi_{\mathrm{lin}})\supset U\to V\subset (\GG, \Pi)$ around $M$. Let $\sigma\in \Bis(\GG, \Pi)$. As $L_\sigma$ is a Poisson map, $L_\sigma\circ \varphi$ linearizes $\Pi$ around $\operatorname{im}\sigma$. Additionally, let $(\mc{V}, \Phi)$ be the chart obtained from $\varphi$ as described above. Then, according to Proposition \ref{prop:linearcoisotropic}, $\tau\in \mc{V}$ is coisotropic if and only if $\Phi(\tau)$ is coisotropic in $\AA$, which happens if and only if $d_{\AA^*}\Phi(\tau)=0$. It follows that $\Phi(\Bis(\GG, \Pi)\cap \mc{V})=\Gamma_c(\AA, d_{\AA^*})\cap \Phi(\mc{V})$. Hence, $\Bis(\GG, \Pi)$ is an embedded submanifold of $\Bis(\GG)$, modelled on $\Gamma_c(\AA, d_{\AA^*})$, that is closed with respect to the compact-open $C^1$-topology (\cite{krieglmichor1997}, 41.9).  This shows that $\Bis(\GG, \Pi)$ is an embedded Lie subgroup. 
	
	By Proposition \ref{prop:coisotropicbisectionsliealgebra}, the Lie algebra of $\Bis(\GG, \Pi)$ with this Lie group structure is exactly $\Gamma_c(\AA, d_{\AA^*})$.
	
For regularity, we use Lemma \ref{lem:subgroupregular}. Consider the map $p:\Bis(\GG)\to \mf{X}^2(\GG)$ given by $p(\sigma)=(L_\sigma)^*(\Pi)$. This map is smooth: if $\sigma_t$ is a smooth curve in $\Bis(\GG)$, then the map $\RR\times \GG \ni (t, g)\mapsto (L_{\sigma_t})^*(\Pi)_g \in \wedge^2 T\GG$ is smooth, so by the discussion in section \ref{subsubsec:topologicalandsmoothstructure}, $L_{\sigma_t}(\Pi)$ is a smooth curve in $\mf{X}^2(\GG)$. From Proposition \ref{prop:coisotropicbisections} \textit{(ii)} it follows that $p^{-1}(0)=\Bis(\GG, \Pi)$. Also, $p$ is constant on the left cosets of $\Bis(\GG, \Pi)$. Hence $\Bis(\GG, \Pi)$ is regular by Lemma \ref{lem:subgroupregular}.
\end{proof}

\subsection{Lagrangian bisections of symplectic groupoids}\label{sec:lagrangianbisections}
When $(\GG, \Omega)\toto(M, \pi)$ is a symplectic groupoid, the Poisson structure is always linearizable around $M$ by Weinstein's Lagrangian neighbourhood theorem. As a special case of Theorem \ref{thm:coisotropicbisectionsliegroup}, we recover the result by Rybicki \cite{Rybicki2001}: the group of Lagrangian bisections $\Bis(\GG, \Omega)$ of a symplectic groupoid is a regular Lie group, whose Lie algebra is given by closed, compactly supported one-forms. Already this group has many interesting connections with the Poisson diffeomorphism group on the base. 

Analogous to Hamiltonian diffeomorphisms on symplectic manifolds, one has the notion of (compactly supported) \textit{exact} Lagrangian bisections $\Bis_{\mathrm{ex},c}(\GG, \Omega)$. These are the time-one flows of (compactly supported) exact time-dependent one-forms $\Omega^1_{\mathrm{ex},c}(M)$. By means of a flux homomorphism for Lagrangian bisections \cite{Rybicki2001, xu1997}, the exact sequence of Lie algebras
\[
0\rightarrow \Omega^1_{\mathrm{ex}, c}(M)\rightarrow \Omega^1_{\mathrm{cl},c}(M) \rightarrow H^1_c(M)\rightarrow 0 
\]
integrates, whenever a certain flux group $\Gamma\subset H^1_c(M)$ is discrete, to an exact sequence of Lie groups
\[
1\rightarrow \Bis_{\mathrm{ex}, c}(\GG, \Omega)\rightarrow \Bis_{c, 0}(\GG, \Omega)\to H^1_c(M)/\Gamma \to 1,
\]
where $\Bis_{c, 0}(\GG, \Omega)$ is the identity component of $\Bis(\GG, \Omega)$, consisting of time-1 flows of (compactly supported, time dependent) closed one-forms. In future work we hope to investigate how the flux homomorphism for Lagrangian bisections as described in \cite{Rybicki2001, xu1997} extends to the setting of this paper.

In light of Proposition \ref{prop:coisotropicbisections}, the Lagrangian bisections act on $(M, \pi)$ by Poisson diffeomorphisms. The map $\pi^\sharp$ induces surjective Lie algebra homomorphisms
\[
\Omega^1_{\mathrm{cl}, c}(M)\to \hamlocc(M, \pi), \quad \Omega^1_{\mathrm{ex}, c}(M)\to \ham_c(M, \pi).
\]
so that it covers the exact sequence
\[
\begin{tikzcd}
	0 \arrow[r] & \Omega^1_{\mathrm{ex}, c}(M) \arrow[r] \arrow[d] & \Omega^1_{\mathrm{cl},c}(M) \arrow[r]\arrow[d] & H^1_c(M) \arrow[r]\arrow[d] & 0\\ 
	0\arrow[r] & \ham_c(M, \pi)\arrow[r] & \hamlocc(M, \pi)\arrow[r] & \pi^\sharp\left(H^1_c(M)\right)\arrow[r] & 0.
\end{tikzcd}
\]
The Lie algebra homomorphisms induced by $\pi^\sharp$ integrate to surjective group homomorphisms
\[
\Bis_{c, 0}(\GG, \Omega)\to \Hamlocc(M, \pi), \quad \Bis_{\mathrm{ex}, c}(\GG, \Omega)\to \Ham_c(M, \pi).
\]
In future work, we hope to investigate the existence of a flux homomorphism on the locally Hamiltonian group for arbitrary Poisson manifolds.

\subsection{Linearization of Poisson groupoids}\label{sec:linearization} Before we proceed to applications of our framework, the main results in \cite{smilde2021linearization} about the linearization of Poisson groupoids are stated for later use. Both theorems linearize integrations of a triangular bialgebroid $(\AA, \AA^*, \pi_\AA)$ (Remark \ref{rk:exactbialgebroids}). The fundamental difference is that the first one is about integrations of $\AA^*$, while the second one concerns integrations of $\AA$. The latter requires the $\AA$-Poisson structure to be of cosymplectic type (Definition \ref{def:cosymplectictype}).
\begin{theorem}[\cite{smilde2021linearization}, Theorem 3.16]\label{thm:dualintegrationlinearizable}
	Let $(\AA, \AA^*, \pi_\AA)$ be a triangular Lie bialgebroid over $M$. Any (local) Poisson groupoid $(\GG^*, \Pi)\toto (M, \pi)$ integrating $(\AA^*, \AA)$ is linearizable around $M$. 
\end{theorem}

\begin{theorem}[\cite{smilde2021linearization}, Theorem 3.23]\label{thm:cosymplecticpairgroupoidlinearizable}
	Let $(\AA, \AA^*, \pi_\AA)$ be a triangular Lie bialgebroid, and assume that $\pi_\AA$ is of $k$-cosymplectic type. Then for any (local) integration $\GG\toto M$ of $\AA$, the Poisson groupoid $(\GG, \lar{\pi_\AA}-\rar{\pi_\AA})$ is linearizable. 
\end{theorem}
In the Appendix \ref{app:scattering} we prove linearization of Poisson groupoids over scattering manifolds, which requires a separate treatment.

\section{Lie groups of Poisson diffeomorphisms}

\subsection{Almost injective Lie algebroids}
Many interesting Lie algebroids are determined by the image of the anchor on sections. 
\begin{definition}
	A Lie algebroid $\AA\Rightarrow M$ is \textit{almost injective} if the anchor is injective on sections. 
\end{definition}
Isomorphisms of almost injective Lie algebroids are uniquely determined by the base map. 
\begin{lemma}
	Let $\AA\Rightarrow M$ be an almost injective Lie algebroid and $f\in \Diff(M)$. Then $Tf$ extends to a (necessarily unique) Lie algebroid isomorphism $(\varphi, f):\AA\to \AA$ if and only if $f_*(\rho(\Gamma_c(\AA)))=\rho(\Gamma_c(\AA))$.  
\end{lemma}
\begin{proof}
One implication is obvious. For the reverse, we first recall that for a vector field $X\in \mf{X}(M)$, the push-forward by $f$ is defined by $(f_*X)_x=T_{f^{-1}(x) }f (X_{f^{-1}(x)})$. Since $\rho$ is injective on the level of sections, there is a unique map $f_*:\Gamma_c(\AA)\to \Gamma_c(\AA)$ determined by $f_* \rho(v)=\rho(f_* v)$. This factors through a $C^\infty(M)$-linear map $\varphi:\Gamma_c(\AA)\to \Gamma_c(f^*\AA)$, defined by $\varphi(v)=(f_*v)\circ f$. Thus, by the Serre-Swan theorem, we obtain bundle maps
			\[
			\begin{tikzcd}
				\AA \arrow[r, "\varphi"] \arrow[d] & f^* \AA \arrow[r]\arrow[d] & \AA \arrow[d]\\
				M \arrow[r, "\mathrm{id}"]   &M \arrow[r, "f"] & M.
			\end{tikzcd}
			\]
			The composite, that we will denote by $(\varphi, f)$, clearly induces $f_*$ on the level of sections. It is invertible because the same construction with $f^{-1}$ gives the inverse. Finally, it is easily checked that this map is a Lie algebroid morphism. 
\end{proof}

Almost injective Lie algebroids are always integrable. Moreover, the holonomy groupoid (the final object in the category of integrations of $\AA\Rightarrow M$) can be recovered from any integration (cf. \cite{debord2000}).
\begin{proposition}[\cite{androulidakiszambon2017}, Proposition 1.9]\label{prop:holonomygroupoid}
	Let $\GG\toto M$ be any integration of an almost-injective Lie algebroid $\AA\Rightarrow M$. Then, as a set, $\Hol(\AA)\toto M$ is given by $\left(\GG/\sim\right) \toto M$, where 
	\[
	g\sim h \iff \substack{ \mbox{there exists local bisections $\sigma_g$ and $\sigma_h$}\\ \mbox{ through $g$ and $h$ such that $t\circ \sigma_g=t\circ \sigma_h$.}}
	\]
	The holonomy groupoid $\Hol(\AA)\toto M$ carries the unique structure of a Lie groupoid for which the map $\GG\to \Hol(\AA)$ is a local diffeomorphism.
\end{proposition}

It follows that the groupoid anchor $\chi:\Hol(\AA)\to M\times M$ induces an injective Lie group homomorphism $\chi_*:\Bis(\Hol(\AA))\to \Bis(M\times M)\cong \Diff(M)$ which differentiates at the identity to the injective map $\rho:\Gamma_c(\AA)\to \mf{X}_c(M)$. 

Recall that the inner automorphisms $\InnAut(\AA)$ of a Lie algebroid $\AA\Rightarrow M$ are given by the time-1 flows of (time-dependent) Lie algebroid sections. If $M$ is not compact, we also want to consider $\InnAut_c(\AA)$, the inner automorphisms of $\AA$ with compact support. For an almost injective Lie algebroid, an automorphism $(\varphi, f):\AA\to \AA$ is uniquely determined by $f$, and can thus be regarded as an element of $\Diff(M)$.
\begin{corollary}\label{cor:innautliesubgroup}
	Let $\AA\Rightarrow M$ be an almost injective Lie algebroid. Then the group $\InnAut_c(\AA)$ is a regular Lie group.
\end{corollary}
\begin{proof}
	The group $\InnAut_c(\AA)$ corresponds to $\Bis_{c, 0}(\Hol(\AA))$, the identity component of the group of compactly supported bisections of $\Hol(\AA)$, which is a regular Lie group. 
\end{proof}

We expect that the inclusion $\InnAut_c(\AA)\hookrightarrow \Diff_c(M)$ is an initial map. However, we could only prove it under additional assumptions, as we will now explain.

\begin{definition}
	Let $\AA\Rightarrow M$ be an almost injective Lie algebroid. We call $\AA$ \textit{isotopically closed} when $\Gamma_c(\AA)$ generates all of its isotopies of inner automorphisms. That is, if an isotopy $\varphi_t$ in $\Diff(M)$ generated by $X_t$ is such that $\varphi_t\in \InnAut_c(\AA)$ for all $t$, then $X_t\in \rho(\Gamma_c(\AA))$ for all $t$. 
\end{definition}

\begin{example} A regular foliation $\FF\subset TM$ is always isotopically closed. More generally, if $\AA\Rightarrow M$ is an almost injective Lie algebroid such that $X\in \rho(\Gamma_c(\AA))$ if and only if $X$ is tangent to the leaves of $\AA$, then $\AA$ is isotopically closed. In particular, the log-tangent bundle of a log-manifold is isotopically closed.
\end{example}
\begin{example}\label{ex:standardisotopicallyclosed}
	If $I$ be a projective divisor ideal (Definition \ref{def:projectivedivisor}), then $T_I M$ is isotopically closed. Indeed, if $\varphi$ is an inner automorphism, then $\varphi$ is the time-1 flow of a (time-dependent) vector field that preserves $I$, and so $\varphi^* I=I$. If $\varphi_t$ is an isotopy of inner automorphisms, smooth in $\Diff(M)$ and generated by $X_t$, then 
	\[
	\frac{d}{dt}\varphi^*_t(I)=\varphi_t^*\left(\LL_{X_t}(I)\right)\subset I,
	\]
	hence $X_t$ is in $\rho(\Gamma(T_I M))$.
	
	In particular, the (normal-crossing) log-tangent bundle (Example \ref{ex:logsymplectic}) and the elliptic tangent bundle (Example \ref{ex:ellipticpoisson}) are isotopically closed. 
\end{example}
\begin{example}
	If $\AA\Rightarrow M$ is an isotopically closed almost injective Lie algebroid, and $\BB\Rightarrow Z$ an (almost injective) isotopically closed subalgebroid over a hypersurface $Z\subset M$, then the rescaling $[\AA;\BB]$ is also isotopically closed. This applies to the scattering-tangent bundle (Example \ref{ex:scatteringsymplectic}).
\end{example}

\begin{remark} In all of the above examples, the anchor also provides an embedding $\rho:\Gamma_c(\AA)\to \mf{X}_c(M)$ of topological vector spaces, and thus it is an initial map. In general, when $\rho(\Gamma_c(\AA))$ is closed in $\mf{X}_c(M)$, the map $\rho$ is an embedding by the open mapping theorem (\cite{jarchow1981}, Theorem 5.5.2) for bornological webbed spaces to inductive limits of Baire topological vector spaces (both of which include inductive limits of Fr\'{e}chet spaces, which is our case) .
\end{remark} 
We expect that every almost injective Lie algebroid is isotopically closed, at least for those for which the anchor $\rho:\Gamma_c(\AA)\to \mf{X}_c(M)$ is an embedding. However, we were not able to prove it in this paper. Therefore, it is added as a separate assumption in the proposition below.

\begin{proposition}\label{prop:innerautomorphismsinitial}
Let $\AA\Rightarrow M$ be an \emph{isotopically closed} almost injective Lie algebroid for which $\rho:\Gamma_c(\AA)\to \mf{X}_c(M)$ is an embedding of topological vector spaces. Then $\InnAut_c(\AA)$ is a regular initial Lie subgroup of $\Diff_{c,0}(M)$. 
\end{proposition}

\begin{proof}
Let $\varphi_\epsilon$ be an isotopy in $\InnAut_c(\AA)$ that is smooth in $\Diff_c(M)$, and generated by $X_\epsilon\in \mf{X}_c(M)$. Since $\AA$ is isotopically closed, it follows that $X_\epsilon=\rho(v_\epsilon)$ for some $v_\epsilon\in \Gamma_c(\AA)$. Because $\rho$ is an embedding, it is initial, and therefore $v_\epsilon$ is a smooth curve in $\InnAut_c(\AA)=\Bis_{c, 0}(\Hol(\AA))$. By regularity of $\Bis(\Hol(\AA))$, $v_\epsilon$ integrates to a smooth curve $\sigma_\epsilon$ in $\Bis_{c, 0}(\Hol(\AA))$. Now, since $\AA$ is almost-injective, the curve $t\circ \sigma_\epsilon$ is generated by $X_\epsilon$, and therefore must agree with $\varphi_\epsilon$. It follows that $\varphi_\epsilon$ is smooth in $\InnAut_c(\AA)$. 
\end{proof}

\begin{remark}\label{rk:anyintegration}
	Let $\GG\toto M$ be any integration of an almost injective Lie algebroid $\AA\Rightarrow M$. Since $\GG\to \Hol(\AA)$ is an isomorphism in a neighbourhood of $M$, the map $\Bis_{c,0}(\GG)\to \Bis_{c,0}(\Hol(\AA))$ is surjective with discrete kernel  
	\[
	K=\ker\left( \Bis_{c, 0}(\GG)\to \Bis_{c, 0}(\Hol(\AA)\right)=\ker \left((\chi_G)_*:\Bis_{c, 0}(\GG)\to \Diff_{c, 0}(M)\right)
	\]
	Although topological discreteness is not enough to prove that $\Bis_{c, 0}(\GG)/K$ is a Lie group (see \cite{krieglmichor1997}, 38.5), there is still an isomorphism of groups
	\[
	\Bis_{c,0}(\GG)/K \to \Bis_{c, 0}(\Hol(\AA))\to \InnAut_c(\AA)
	\]
	and therefore $\Bis(\GG)/K$ admits the structure of a convenient Lie group this way.
\end{remark}

\begin{example}
	If the anchor of $\FF\Rightarrow M$ is injective already on the level of the vector bundles, it corresponds to a regular (i.e. non-singular) foliation. In this case, there is an identification $\InnAut_c(\FF)=\Fol_{c,0}(\FF)$, where $\Fol(\FF)$ is the group of diffeomorphisms sending each leaf to itself, that equips $\Fol(\FF)$ with the structure of a regular initial Lie subgroup of $\Diff(M)$. It coincides with the one described in \cite{Rybicki2001foliated}. 
\end{example}

\begin{remark}
	Even for a (non-singular) foliation $\FF\Rightarrow M$, the subgroup $\InnAut(\FF)=\Fol(\FF)$ may not be embedded in $\Diff(M)$. Consider for instance a foliation on the torus $T^2$ with coordinates $(\theta_1, \theta_2)$ whose leaves are integral curves of a constant vector field $X=\del_{\theta_1}+a\del_{\theta_2}$, where $a\in \RR$ is irrational. For each $t\in \RR$, the map $\varphi_t(\theta_1, \theta_2)=(\theta_1+t, \theta_2)$ preserves the foliation, and for a dense subset of $\RR$, the maps $\varphi_t$ are foliated. Hence, there is a sequence $\varphi_{t_n}\in \Fol(\FF)$ such that $\varphi_{t_n}$ converges to the identity in $\Diff(M)$, but does not converge in $\Fol(\FF)$. 
\end{remark}

Another interesting application arises when we consider almost everywhere non-degenerate Poisson structures, combined with the Lagrangian bisections of its integrating symplectic groupoid. 

\begin{theorem}\label{thm:hamlocinitial}
	Let $(M, \pi)$ be a Poisson manifold whose non-degeneracy locus is open and dense. Then the group $\Hamlocc(M, \pi)$ is a regular Lie group, which is an initial subgroup of $\Diff(M, \pi)$ whenever the cotangent algebroid $T^*M$ is isotopically closed and $\pi^\sharp:\Gamma_c(T^*M)\to \mf{X}_c(M)$ is an embedding. 
	
	More generally, if $(\AA, \pi_\AA)\Rightarrow M$ is an almost injective Poisson Lie algebroid and $\pi_\AA$ is generically non-degenerate, then $\Hamlocc(\AA, \pi_\AA)$ is a regular Lie group and and initial subgroup of $\Diff(M, \pi)$ whenever $\AA^*$ is isotopically closed and $\rho_\AA\circ\pi_\AA^\sharp:\Omega^1_c(\AA)\to \mf{X}_c(M)$ is an embedding. 
\end{theorem}
\begin{proof}
	We prove the general case. First, the dual $\AA^*$ becomes an almost-injective Lie algebroid as in Remark \ref{rk:exactbialgebroids}. Let $(\GG^*, \Pi)\toto (M, \pi)$ be a Poisson groupoid integrating the bialgebroid $(\AA^*, \AA)$. As in Remark \ref{rk:anyintegration}, we set
	\[
	K=\ker\left( (\chi_{\GG^*})_*:\Bis_{c, 0}(\GG^*)\to \Diff(M)\right),
	\]
	which is a discrete subgroup of $\Bis_{c, 0}(\GG^*)$. Since this is the kernel of a smooth covering map of Lie groups, we can find a neighborhood $\mathcal{W}$ of the identity in $\Bis_{c, 0}(\GG^*)$ satisfying
	\[
	\mc{W}=\mc{W}^{-1}, \quad \mc{W}\mc{W}\cap k\mc{W}\mc{W}=\emptyset \quad \mbox{for all $k\in K$}.
	\]
	 By Theorem \ref{thm:dualintegrationlinearizable}, the Poisson groupoid $(\GG^*, \Pi)$ is linearizable, and therefore the identity component $\Bis_{c,0}(\GG^*, \Pi)$ of the group of coisotropic bisections is an embedded submanifold of $\Bis(\GG^*)$ (Theorem \ref{thm:coisotropicbisectionsliegroup}), it follows that 
	\[
	K_{\Pi}=K\cap \Bis_{c, 0}(\GG^*, \Pi)
	\]
	is discrete in $\Bis_{c,0}(\GG^*, \Pi)$. Moreover, shrinking $\mc{W}$ if necessary, we have, by setting $\mc{W}_{\Pi}=\mc{W}\cap \Bis(\GG^*, \Pi)$, an open neighborhood of the identity satisfying
	\[
	\mc{W}_\Pi=\mc{W}_\Pi^{-1}, \quad \mc{W}_\Pi\mc{W}_\Pi\cap k\mc{W}_\Pi\mc{W}_\Pi=\emptyset \quad \mbox{for all $k\in K_\Pi$}.
	\]
	By \cite{krieglmichor1997}, 38.5, the quotient $\Bis_{c, 0}(\GG^*, \Pi)/K_\Pi \cong \Hamlocc(\AA, \pi_\AA)$ is then a regular Lie group.
	
	Under the additional assumptions, it becomes an initial Lie subgroup of $\Diff_c(M)$, because $\InnAut_c(\AA^*)$ is. 
\end{proof}

\subsection{Poisson structures of divisor type}\label{sec:divisortypdiffeos} Our approach is particularly effective for a class of generically non-degenerate Poisson structures, namely those of divisor type considered \cite{klaasse2018} (see Section \ref{sec:divisortype}).
Let $(M, \pi)$ be a Poisson structure of divisor type with projective divisor ideal $I_\pi$. The algebroid $T_{I_\pi}M$ is almost injective, and by Proposition \ref{prop:poissonprojectivedivisor}, every Poisson vector field is a section of $T_{I_\pi}M$. When $T_{I_\pi}^* M$ locally admits bases of closed sections, the Poisson structure $\pi$ lifts to a Poisson structure $\pi_I$ on $T_{I_\pi}M$ by Theorem 4.35 in \cite{klaasse2018}. If in addition $I_\pi$ is a standard ideal, the lift $\pi_I$ is non-degenerate. 

\begin{theorem}\label{thm:poissondivisortype}
	Let $(M, \pi)$ be a Poisson structure of standard divisor type. Suppose that $T_{I_\pi}M$ locally admits bases of closed sections. Then $\Diff(M, \pi)$ is a regular initial Lie subgroup of $\Diff(M)$.  
\end{theorem}
\begin{proof}
	We will equip the path component $\Diff_{c, 0}(M, \pi)$ of $\Diff(M,\pi)$ with the structure of a regular initial Lie subgroup of $\Diff_{c, 0}(M)$. By Proposition \ref{prop:poissonprojectivedivisor}, every Poisson diffeomorphism in $\Diff_{c, 0}(M, \pi)$ is an inner automorphism of $T_{I_\pi}M$, to which $\pi$ lifts to a non-degenerate Poisson structure $\pi_I$. By Example \ref{ex:standardisotopicallyclosed}, the Lie algebroid $T_{I_\pi}M$ is isotopically closed, and therefore $\InnAut_c(T_{I_\pi}M, \pi_I)$ is a regular initial Lie subgroup of $\Diff_{c,0}(M)$ by Theorem \ref{thm:almostinjectivePoisson}. Now, the whole group $\Diff(M, \pi)$ becomes a Lie group by declaring $\Diff_c(M, \pi)$ to be open $\Diff(M, \pi)$ and using group multiplication to translate the manifold structure to the entire group. This way, $\Diff(M, \pi)$ is becomes an initial subgroup of $\Diff(M)$. 
\end{proof}

\subsubsection{Log-symplectic structures}

Let $(M, Z)$ be a (normal-crossing) log-manifold (Example \ref{ex:logsymplectic}). Analogous to the boundary of a manifold with corners, the hypersurface $Z$ comes with a natural stratification $Z=\cup_{j=0}^{2n-1}Z_j$ (the $Z_j$'s are (unions of) the $j$-dimensional orbits of $T_Z M$). Note that by the local form, each $Z_j$ is embedded (but not closed!). Because $(M\setminus \left(\cup_{i=1}^{2n-2}Z_i\right), Z_{2n-1})$ is a smooth log-manifold, the log-symplectic structure induces a cosymplectic structure on $i:Z_{2n-1}\hookrightarrow M$, whose one-form is given by $\alpha=i^*(\iota_\EE \omega)$, where $\EE$ is the Euler vector field obtained via any tubular neighbourhood of $Z_{2n-1}$ (\cite{guilleminmirandapires2014}, Proposition 10).  
\begin{lemma}\label{lem:loglocallyhamiltonianfoliated}
	Let $(M, Z, \omega)$ be a log-symplectic manifold with underlying Poisson structure $\pi$. A vector field $X\in \mf{X}(M)$ is tangent to the symplectic foliation of $\pi$ if and only if there is a one-form $\eta\in \Omega^1(M)$ such that $\pi^\sharp(\eta)=X$. 
\end{lemma}
\begin{proof}
	Any vector field of the form $\pi^\sharp(\eta)$ it tangent to the symplectic foliation. Conversely, if $X$ is tangent to the symplectic foliation, it is certainly tangent to $Z$ and thus can be regarded as a section of $T_Z M$, and thus we can set $\eta=\omega^\flat(X)$.  Since $X$ is tangent to the symplectic foliation, it must be in the kernel of $\alpha=i^*(\iota_\EE\omega)$ when restricted to $Z_{2n-1}$. In local coordinates $(x_i)$ adapted to $Z$, we have $Z_{2n-1}=\{ \prod_{i=1}^k x_i=0: \mbox{no two $x_i$'s are zero}\}$. Therefore, on $Z_{2n-1}\cap \{x_i=0\}$, the Euler vector field is given by $x_i\del_{x_i}$. We find that
	\[
	0=\iota_X\alpha=\omega\left( x_i\deldel{x_i}, X\right)\big\vert_{Z_{2n-1}\cap\{x_i=0\}}=-\eta\left(x_i\deldel{x_i}\right)\big\vert_{Z_{2n-1}\cap\{x_i=0\}}.
	\]
	By continuity, $\iota_{x_i\del_{x_i}}\eta=0$ along $\{x_i=0\}$. Since this is true for all $i=1, \dots, k$, it follows that $\eta\in \Omega^1(M)$. 
\end{proof}
Consequently, for log-symplectic manifolds, the locally Hamiltonian and foliated vector fields coincide! This implies that $\Hamlocc(M, \pi)$ is the identity component of $\Fol_c(M, \pi)$. 
\begin{theorem}\label{thm:logsymplecticliegroup}
	Let $(M, Z)$ be a log-manifold. Then 
	\[
	\Diff(M, Z)=\{ \varphi\in \Diff(M): \varphi(Z)=Z\}
	\]
	is a regular initial Lie subgroup of $\Diff(M)$ with Lie algebra $\mf{X}_c(M, Z)$.

	Let $(M, Z, \omega)$ be a log-symplectic manifold with underlying Poisson structure $\pi$. Then $\Diff(M, \pi)$ is a regular embedded Lie subgroup of $\Diff(M, Z)$, with Lie algebra $\mf{X}_c(M, \pi)$. 
	
	Finally, $\Hamlocc(M, \pi)$ coincides with the identity component of $\Fol(M, \pi)$, and comes with the structure of a regular Lie group for which it is initial in $\Diff(M, \pi)$, with Lie algebra $\hamlocc(M, \pi)=\mf{fol}_c(M, \pi)$. 
\end{theorem}
\begin{remark}
	If $Z$ is a smooth hypersurface, Ebin and Marsden proved in \cite{ebinmarsden1970} that $\Diff(M, Z)$ is actually embedded in $\Diff(M)$ by means of a metric for which $Z$ is geodesically closed. This is not immediate from our approach.
\end{remark}

\begin{proof}
	Let $\GG\toto M$ be the holonomy groupoid of $T_ZM$. The identity path-component of $\Diff(M, Z)$ identifies with the (compactly supported) inner automorphisms of $T_ZM$, which is an initial subgroup of $\Diff(M)$ by Proposition \ref{prop:innerautomorphismsinitial} after identifying $\Diff_{c, 0}(M, Z)=\Bis_{c, 0}(\GG)$. The second statement is a special case of Theorem \ref{thm:poissondivisortype}. Finally, due to Lemma \ref{lem:loglocallyhamiltonianfoliated}, the identity component of $\Fol(M,\pi)$ coincides with $\Hamlocc(M, \pi)$. The cotangent algebroid $T^*M$ is almost injective an isotopically closed by Lemma \ref{lem:poissonisotopy}, and therefore $\Hamlocc(M, \pi)$ obtains the structure of a regular initial Lie subgroup of $\Diff(M)$ via Theorem \ref{thm:hamlocinitial}.
\end{proof}

\begin{remark}
	When $Z\subset M$ is a smooth hypersurface, explicit integrations of the log-tangent bundle $T_ZM$ have been constructed in \cite{gualtierili2014}. One of its integrations, that we call the \textit{log-pair} groupoid $\Pair_Z(M)$, can be obtained by blowing up $Z\times Z$ in $M\times M$. After removing a certain subset (namely the arrows that point into and out of $Z$), $\Pair_Z(M)\toto M$ becomes a groupoid integrating $T_ZM$. If $\omega$ is a log-symplectic structure on $T_ZM$, then $\Pi^{-1}=s^*\omega-t^*\omega$ is a log-symplectic structure on $p_Z^!\Pair_Z(M)=T_{\mc{H}}\Pair_Z(M)$ (see Proposition \ref{prop:pullbackscatteringb}). For the resulting Poisson groupoid $(\Pair_Z(M), \Pi)$, the anchor to $M\times M$ induces isomorphisms $\Bis(\Pair_Z(M))\cong \Diff(M, Z)$ and $\Bis(\Pair_Z(M), \Pi)\cong \Diff(M, \pi)$. Thus in this case, there is a very explicit and concrete Poisson groupoid whose group of coisotropic bisections corresponds exactly to the group of Poisson diffeomorphisms.
\end{remark}

\subsubsection{Elliptic symplectic structures}

Let $|D|$ be an elliptic divisor on $M$. If $X\in \mf{X}(M)$ is a vector field with the property that $[X, Y]$ is elliptic for all $Y\in \mf{X}(M, |D|)$, then $X\in \mf{X}(M, |D|)$. This implies that if $\varphi_t$ is an isotopy generated by $X_t$ consisting of automorphisms of $T_{|D|}M$, then $X_t$ is an elliptic vector field. It follows that $\Diff_{c,0}(M, |D|):=\operatorname{Aut}_{c, 0}(T_{|D|}M)\cong \InnAut_c(T_{|D|}M)$.

The symplectic foliation of an elliptic symplectic manifold $(M, \pi)$ depends on the elliptic residue of the elliptic two-form $\omega\in \Omega^2(T_{|D|}M)$, defined as follows. Near the degeneracy locus $i:D\hookrightarrow M$, we can write
\[
\omega=\lambda d\log r\wedge d\theta+d\log r \wedge \alpha_1+ d\theta\wedge\alpha_2 + \beta
\]
where $\lambda\in \Omega^0(D;\mf{k}^*)$, with $\mf{k}=\det \left( \ker \rho_{T_{|D|}M}\vert_D\right)$, is locally constant. Then $\lambda$ is the \textit{elliptic residue} of $\omega$.

\begin{definition}
	Let $(M, \pi)$ be an elliptic symplectic manifold with $\lambda$ defined as above. If $\lambda\neq 0$, we say that $(M, \pi)$ has \textit{non-zero elliptic residue}. If $\lambda=0$, then we say that $(M, \pi)$ has \textit{zero} or \textit{vanishing elliptic residue}.
\end{definition}

 If $(M, \pi)$ has non-zero elliptic residue, then the symplectic foliation of $(M, \pi)$ consists of (the connected components of) $M\setminus D$ and $D$. Therefore, every Poisson vector field is automatically foliated. 

When the elliptic residue vanishes, the symplectic foliation is given by connected components of $M\setminus D$ and the foliation induced by the kernels of the closed one-forms $i^*\alpha_1, i^*\alpha_2\in \Omega^1(D)$ (these are the radial and $\theta$-residues, respectively).

\begin{lemma}\label{lem:ellipticfoliated}
Let $(M, \pi)$ be an elliptic symplectic manifold, $|D|$ its elliptic divisor and $\omega\in \Omega^2(T_{|D|}M)$ its elliptic symplectic form. Assume $\omega$ has zero elliptic residue. Then every foliated Poisson vector field is locally Hamiltonian.
\end{lemma}
\begin{proof}
	This proof is inspired by the proof of Lemma 1.10 in \cite{Cavalcantigualtieri2017}. Locally, in polar coordinates $(r, \theta)$ around $D$, we can write
	\[
	\omega=d\log r \wedge \alpha_1 +d\theta \wedge \alpha_2 + \beta.
	\]
	An elliptic vector field $X\in \mf{X}(M, |D|)$ is Poisson if and only if $\iota_X \omega\in \Omega^1(T_{|D|}M)$ is closed and it is foliated if and only if $\iota_X\alpha_1$ and $\iota_X\alpha_2$ vanish over $D$. If $X\in \mf{fol}(M, \pi)$, we have to show that $\iota_X\omega$ comes from a de Rham form on $M$. It is enough to show this in the locally. 

Let $\eta=\iota_X\omega\in \Omega^1(T_{|D|}M)$ be the corresponding closed elliptic form. The first step is to show that $\eta$ is cohomologous to a smooth de Rham form. As in the proof of Lemma 1.10 in \cite{Cavalcantigualtieri2017}, we denote by $\rho_t:U\to U$ for $t\in S^1$ the $S^1$-action in a tubular neighbourhood $U$ of $D$, generated by $\del_\theta$. Since this is an elliptic vector field, $\rho^*_t$ acts trivially on the level of the cohomology of $T_{|D|} M$.
\begin{claim*}
	The form $\rho^*_t\eta$ defines a foliated Poisson vector field for all $t\in S^1$.
\end{claim*}
\begin{claimproof}
This is a direct computation. The form $\rho^*_t\eta$ is certainly closed. It is foliated when the associated vector field $(\omega^\flat)^{-1}(\rho^*_t\eta)$ paired with $\alpha_i$ vanishes over $D$. For $\alpha_1$, this can be shown as follows. First, we observe that
\[
\rho^*_t\omega=d\log r\wedge \rho^*_t \alpha_1+d\theta \wedge \rho^*_t \alpha_2 +\rho^*_t \beta.
\]
Then 
\begin{align*}
	\langle \alpha_1, (\omega^\flat)^{-1}(\rho^*_t \eta) \rangle&=-\langle (\omega^\flat)^{-1}(\alpha_1), \rho^*_t\eta \rangle = -\langle r\del_r, \rho^*_t (\iota_X \omega) \rangle \\
	&= - \langle r\del_r, \iota_{\rho^*_t X} (\rho^*_t\omega) \rangle= (\rho^*_t\omega)(r\del_r, \rho^*_t X)\\
	&=\langle \rho^*_t \alpha_1, \rho^*_t X\rangle =\rho^*_t \langle \alpha_1, X\rangle,
\end{align*}
which vanishes over $D$ because $\rho_t$ leaves $D$ invariant. The computation with $\alpha_2$ is the same. This proves the claim.
\end{claimproof}
By averaging the forms $\rho^*_t\eta$ over $S^1$, we obtain an $\theta$-invariant form $\overline{\eta}$ for which $\overline{X}=(\omega^\flat)^{-1}(\overline{\eta})$ is foliated. Writing
\[
\overline{\eta}=f_1 d\log r+ f_2d\theta + \zeta
\]
for smooth functions $f_0, f_1$ on $U$ and $\zeta\in \Omega^1(U)$, we have that
\begin{itemize}[noitemsep, topsep=0em]
	\item $f_1$ and $f_2$ are $S^1$-invariant, and
	\item $f_i=\iota_{\overline{X}}\alpha_i$, so it vanishes over $D$.
\end{itemize}
Therefore, $f_1$ and $f_2$ are divisible by $r^2$ and thus $\overline{\eta}$ is a smooth one-form on $U$. 

Finally, since $\overline{\eta}$ is cohomologuous to $\eta$, we can write $\eta=\overline{\eta}+df$. Since every exact elliptic one-form comes from a smooth one-form on $M$, it follows that $\eta$ is indeed a smooth one-form on $U$ (in general, if $\AA\Rightarrow M$ is a Lie algebroid, with differential $d_\AA$, then $d_\AA f=\rho_\AA^* (df)$ for any $f\in C^\infty(M)$). This completes the proof.
\end{proof}

With this lemma out of the way, we have a complete picture of the Poisson diffeomorphism group of elliptic symplectic manifolds.

\begin{theorem}
	Let $|D|$ be an elliptic divisor. Then $\Diff(M, |D|)$ is a regular initial Lie subgroup of $\Diff(M)$ with Lie algebra $\mf{X}(M, |D|)$.
	
	Let $(M, \pi)$ be an elliptic symplectic manifold, with elliptic divisor $|D|$. Then the groups $\Diff(M, \pi)$, $\Fol(M, \pi)$ and $\Hamlocc(M, \pi)$ are regular initial Lie subgroups of $\Diff(M)$ with Lie algebras $\mf{X}_c(M, \pi)$,  $\fol_c(M, \pi)$ and $\hamlocc(M, \pi)$, respectively.
	
	If $(M, \pi)$ has non-zero elliptic residue, then $\Fol(M, \pi)$ is an open subgroup of $\Diff(M, \pi)$. If $(M, \pi)$ has zero elliptic residue, then the identity component $\Fol_{c,0}(M, \pi)$ of $\Fol(M, \pi)$ coincides with $\Hamlocc(M, \pi)$.
\end{theorem}

\begin{proof}
	Since $T_{|D|}^*M$ locally admits bases of closed sections, the following is a special case of Proposition \ref{prop:innerautomorphismsinitial} and Theorem \ref{thm:poissondivisortype}.
	The group $\Hamlocc(M, \pi)$ carries the structure of an initial Lie subgroup of $\Diff(M, \pi)$ by Theorem \ref{thm:hamlocinitial}.
	
	Finally, if the elliptic residue is non-zero, then $\Fol(M, \pi)$ coincides with the subgroup of $\Diff(M, \pi)$ that sends (the components of $D$) to itself, which is open in $\Diff(M, \pi)$. If the elliptic residue is zero, then by Lemma \ref{lem:ellipticfoliated} every foliated vector field is locally Hamiltonian, so we can equip $\Fol(M, \pi)$ with the structure of a Lie group by declaring $\Hamlocc(M, \pi)$ to be open.
\end{proof}

\subsection{Scattering-symplectic structures}
The following is a consequence of Theorem \ref{thm:scatteringintegrationlinearizable}.

\begin{theorem}
	Let $(M, Z, \omega)$ be a scattering-symplectic manifold with underlying Poisson structure $\pi$. Then $\Diff(M, \pi)$ is a regular initial Lie subgroup of $\Diff(M)$, with Lie algebra $\mf{X}_c(M, \pi)$. 
\end{theorem}

\begin{proof}
	We equip the path component $\Diff_{c, 0}(M, \pi)$ with a Lie group structure as follows. Any Poisson vector field $X$ must be tangent to $Z$, hence it follows $\Diff_{c, 0}(M, \pi)\subset \Diff_{c, 0}(M,Z)$. Recall from the proof of Theorem \ref{thm:logsymplecticliegroup} that $\Diff_{c, 0}(M, Z)$ obtains its Lie group structure by identifying it with $\Bis_{c,0}(\GG)$, with $\GG\toto M$ the holonomy groupoid of $T_ZM$.  By Theorem \ref{thm:scatteringintegrationlinearizable}, the groupoid $\GG\toto M$ comes with a Poisson structure $\Pi$ that is linearizable around $M$, Therefore, Theorem \ref{thm:coisotropicbisectionsliegroup} says that $\Bis_{c, 0}(\GG, \Pi)$ is a regular embedded Lie subgroup of $\Bis_{c, 0}(\GG)$. Clearly, $\Bis_{c, 0}(\GG, \Pi)$ identifies with $\Diff_{c, 0}(M, \pi)$. 
\end{proof}

\subsection{Almost injective (co)symplectic Lie algebroids}  An \textit{inner $\AA$-Poisson automorphism} of a Poisson algebroid $(\AA, \pi_\AA)\Rightarrow M$ is an inner automorphism $\varphi\in \InnAut(\AA)$ that preserves the bivector $\pi_\AA$. This gives rise to the subgroup $\InnAut(\AA, \pi_\AA)$ of $\InnAut(\AA)$. Its infinitesimal generators correspond precisely to Poisson sections, which are the sections of $\AA$ closed with respect to the Lie algebroid structure on $\AA^*$.
\begin{proposition}
	Let $(\varphi_t, \Phi_t): \AA\to \AA$ be the flow of a time-dependent section $v_t\in \Gamma(\AA)$. Then $\varphi_t$ is a path in $\InnAut(\AA,\pi_\AA)$ if and only if $[\pi_\AA, v_t]=0$ for all $t$.
\end{proposition}

We can now prove the following general result. 

\begin{theorem}\label{thm:almostinjectivePoisson}
Let $(\AA,\pi_\AA)\Rightarrow M$ be an almost-injective Poisson algebroid of cosymplectic type. The group $\InnAut_c(\AA, \pi_\AA)$ is a regular embedded Lie subgroup of $\InnAut_c(\AA)$ with Lie algebra $\Gamma_c(\AA, \pi_\AA)$.

If in addition $\AA$ is isotopically closed and $\rho:\Gamma_c(\AA)\to \mf{X}_c(M)$ is initial, then $\InnAut_c(\AA,\pi_\AA)$ is an initial Lie subgroup of $\Diff(M)$. 
\end{theorem}
\begin{proof}
	The holonomy groupoid $\Hol(\AA)$ integrates $\AA$, and the Poisson structure $\lar{\pi_\AA}-\rar{\pi_\AA}$ on $\Hol(\AA)$ is linearizable by Theorem \ref{thm:cosymplecticpairgroupoidlinearizable}.
	
	If $\sigma\in \Bis(\Hol(\AA))$, then its image in $(M\times M, \pi\times (-\pi))$ is coisotropic precisely when $\sigma$ is coisotropic in $(\Hol(\AA), \lar{\pi_\AA}-\rar{\pi_\AA})$, because $\AA$ is almost-injective. This implies that $t\circ\sigma$ induces an $\AA$-Poisson isomorphism.
	
	Under the additional assumptions, it becomes initial by Proposition \ref{prop:innerautomorphismsinitial}. 
\end{proof}

\subsection{Poisson structures of cosymplectic type} In this section we consider the symmetry group of Poisson stuctures of cosymplectic type on the tangent bundle of a manifold. 

\begin{lemma}
Let $(M, \pi)$ be a Poisson manifold of cosymplectic type, with underlying foliation $\FF$. Then $\FF$ has trivial holonomy. 
\end{lemma}
\begin{proof}
	Choose a cosymplectic structure $(\alpha_1,\dots, \alpha_k, \omega)$ inducing $\pi$, The Reeb vector fields $R_1,\dots, R_k$ are pairwise commuting and transverse to $\FF$, while also preserving $\FF$. Let $L$ be a leaf of $\FF$ and $x\in L$ and $\gamma:[0,1]\to L$ a loop based at $x$. The flows $\varphi^i_{\epsilon}$ of $R_i$ induce an embedding $(\epsilon_1, \dots, \epsilon_k)\mapsto \varphi^1_{\epsilon_1}\circ \dots \circ \varphi^k_{\epsilon_k}(x)$ of a small disc through $x$ transverse to $\FF$. Since the image of $\gamma$ is compact, for $\epsilon_i$ small enough, $\varphi_{\epsilon_1}^1\circ \dots\circ \varphi^k_{\epsilon_k}\circ\gamma$ is a loop tangent to $\FF$ starting at $\varphi_{\epsilon_1}^1\circ \dots\circ \varphi^k_{\epsilon_k}(x)$. Clearly, at $t=1$, it has returned to its starting point. 
\end{proof}
It follows that $\Hol(\FF)$ coincides with the relation groupoid \[
\operatorname{Rel}(\FF)=\{(x, y)\in M\times M: \mbox{$x$ and $y$ belong to the same leaf}\}.
\]
We aim to strengthen Theorem \ref{thm:almostinjectivePoisson} in the case that $\AA=TM$. 
\begin{theorem}\label{thm:cosymplecticliegroup}
	Let $(M, \pi)$ be a Poisson manifold of cosymplectic type with underlying foliation $\FF$. Then $\Diff(M, \pi)$ is an embedded regular Lie subgroup of $\Diff(M)$. Furthermore, $\Fol(M, \pi)$ is an initial Lie subgroup of $\Diff(M, \pi)$.
	
	If the relation groupoid $\operatorname{Rel}(\FF)$ is embedded in $M\times M$, then $\Fol(M, \pi)$ is embedded in $\Diff(M, \pi)$.  
\end{theorem}
Note that the relation groupoid $\operatorname{Rel}(\FF)$ is embedded in $M\times M$ when every symplectic leaf of $(M,\pi)$ is embedded.

Let $(M, \pi)$ be a Poisson manifold of cosymplectic type, and choose a cosymplectic structure $(\alpha_1, \dots, \alpha_k, \omega)$ inducing $\pi$. To obtain charts of $\Diff(M, \pi)$ adapted to $\Fol(M, \pi)$, it is crucial to keep track of the one-forms that are part of the cosymplectic structure in the linearization of the pair Poisson groupoid $(M\times M, \lar{\pi}-\rar{\pi})$ (this information is lost in the proof of Theorem \ref{thm:cosymplecticpairgroupoidlinearizable}), because of the following observation.

\begin{lemma}\label{lem:relationgroupoidleaf}
Suppose that $M$ is connected. Let $\alpha_1, \dots, \alpha_k$ be everywhere independent closed one-forms on $M$ and set $\FF=\cap_i \ker \alpha_i$. The relation groupoid $\operatorname{Rel}(\FF)$ is the leaf of the foliation $\cap_{i}\ker(\rar{\alpha_i}-\lar{\alpha_i})$ on $M\times M$ through the diagonal. 
\end{lemma}

To keep track of the cosymplectic structure, we identify $N\Delta=\{(-v, v): v\in TM\} \subset TM\times TM$ and choose a tubular neighbourhood $\psi:TM\to U\subset M\times M$ of $\Delta$ that respects the splittings along $M$, meaning that $T\psi$ sends $(v, w)\in T_x(TM)\cong T_xM\oplus T_xM$ to $(v-w, v+w)\in T_{(x, x)}(M\times M)$. We let $p:T^*M\to M$ be the bundle projection, and denote by $\ell_{R_i}:T^*M\to \RR$ the linear function associated to the Reeb vector field $R_i$: $\ell_{R_i}(\alpha)=\alpha(R_i(x))$ for $\alpha\in T^*_xM$.

\begin{lemma}\label{lem:diagonalcosymplecticstructure}
	Let $\flat:TM\to T^*M$ be the flat map induced by the cosymplectic structure $(\alpha_1, \dots, \alpha_k, \omega)$, with inverse $\sharp=\flat^{-1}$. Then the following equalities hold along $M$:
	\[
		p^*\alpha_i\vert_M=\sharp^*\psi^*\left(\frac{1}{2}\left(\rar{\alpha_i}+\lar{\alpha_i}\right)\right)\big\vert_M, \quad
		d\ell_{R_i}\big\vert_M=\sharp^*\psi^*\left(\frac{1}{2}\left(\rar{\alpha_i}-\lar{\alpha_i}\right)\right)\big\vert_M, 
		\]
		\[
		\left(\omega_{\can}-\sum_{i=1}^k d(\ell_{R_i} p^*\alpha_i) \right)\big\vert_M=\sharp^*\psi^*\left(\frac{1}{2} \left(\rar{\omega}-\lar{\omega}\right)\right)\big\vert_M. 
	\]
\end{lemma}
\begin{proof}
	Since the map $\sharp$ is a vector bundle map, its derivative along the zero section is given in terms of the natural splittings by $T\sharp(v, w)=(v, \sharp(w))$.
	
	For the first two equalities, we have on one hand, for $(v, \beta)\in T_x(T^*M)$:
	\[
	p^*\alpha_i(v, \beta)=\alpha_i(v), \quad d\ell_{R_i}(v, \beta)=\beta(R_i),
	\]
	while the right side evaluates to:
	\begin{align*}
		\sharp^*\psi^*\left(\frac{1}{2}\left(\rar{\alpha_i}+\lar{\alpha_i}\right)\right)(v, \beta)&=\frac{1}{2}\left(\rar{\alpha_i}+\lar{\alpha_i}\right)(v-\sharp(\beta), v+\sharp(\beta))=\alpha_i(v), 
	\end{align*}
and the second equality goes similar. For the third one, let $(v, \beta), (w, \gamma)\in T_x(T^*M)$. Then,
\begin{align*}
	\sharp^*\psi^*\left(\frac{1}{2}\left(\rar{\omega}-\lar{\omega}\right)\right)\left( ( v, \beta), (w, \gamma)\right) &=\frac{1}{2}\left(\rar{\omega}-\lar{\omega}\right)\left( T\psi(v, \sharp(\beta)), T\psi (w, \sharp(\gamma))\right)\\
	&=\omega(\sharp(\beta), w)+\omega(v, \sharp(\gamma))\\
	&=\beta(w)-\gamma(v)-\left(\sum_i \beta(R_i)\alpha_i(w)-\gamma(R_i)\alpha_i(v)\right), 
\end{align*}
where we used that $\iota_{\sharp(\beta)}\omega=\beta-\sum_i \beta(R_i)\alpha_i$.
On the other hand,
\begin{align*}
	\omega_{\can}((v, \beta), (w, \gamma))&=\beta(w)-\gamma(v)\\
	d\ell_{R_i}\wedge p^*\alpha_i\left( (v, \beta), (w, \gamma)\right)&=\beta(R_i)\alpha_i(w)-\gamma(R_i)\alpha_i(v). 
\end{align*}
The third equality now follows easily.
\end{proof}
\begin{remark}\label{rk:cotangentcosymplectic}
	The cosymplectic structure $(p^*\alpha_1, d\ell_{R_1}, \dots p^*\alpha_k, d\ell_{R_k}, \omega_{\can}-\sum_i d(\ell_{R_i}p^*\alpha_i))$ induces a linear Poisson structure on $T^*M$, which is Poisson diffeomorphic via the flat map to the linear Poisson structure on $TM$ dual to the cotangent algebroid. Using the cosymplectic structure to write $T^*M=\FF^*\times \RR^k$, we see that a section $(\eta, f)\in \Gamma(\FF^*\times \RR^k)$ has a coisotropic image in $T^*M$ if and only if $d_\FF f=0$ and $d_\FF\eta=0$. This is equivalent to $\sharp(\eta, f)$ being a Poisson vector field.  
\end{remark}
\begin{proof}[Proof of Theorem \ref{thm:cosymplecticliegroup}]
	The first part is just Theorem \ref{thm:almostinjectivePoisson} in the case that $\AA=TM$. Alternatively, the Poisson pair groupoid is linearizable by Theorem \ref{thm:cosymplecticpairgroupoidlinearizable}, and so by Theorem \ref{thm:coisotropicbisectionsliegroup} its subgroup of coisotropic bisections is embedded. Its group of foliated diffeomorphisms is an initial Lie subgroup by Theorem \ref{thm:almostinjectivePoisson}, taking $\AA=\FF$. 
	
	Suppose now that $\operatorname{Rel}(\FF)$ is embedded in $M\times M$. According to Lemma \ref{lem:diagonalcosymplecticstructure} and the Moser lemma for cosymplectic Lie algebroids (\cite{smilde2021linearization}, Lemma 3.1), we can find
	\begin{itemize}[noitemsep, topsep=0em]
		\item open neighbourhoods $U\subset T^*M$ of $M$ and $V\subset M\times M$ of the diagonal;
		\item a Poisson diffeomorphism $\varphi:U\to V$ such that $\varphi^*\left( \frac{1}{2} \left(\rar{\alpha_i}+\lar{\alpha_i}\right)\right)=p^*_*\alpha_i$ and $\varphi^*\left( \frac{1}{2}\left( \rar{\alpha_i}-\lar{\alpha_i}\right)\right) = d \ell_{R_i}$. 
	\end{itemize} 
Identify $T^*M=\FF^*\times \RR^k$ as in Remark \ref{rk:cotangentcosymplectic}. Since $\operatorname{Rel}(\FF)$ is embedded, and $\varphi^*\left(\frac{1}{2}\left( \rar{\alpha_i}-\lar{\alpha_i}\right)\right)=d\ell_{R_i}$, we can assume by Lemma \ref{lem:relationgroupoidleaf}, after possibly shrinking $U$, that $\varphi(\eta, f)\in \operatorname{Rel}(\FF)$ if and only if $f=0$. In this tubular neighbourhood, the $C^1$-open
\[
\mc{U}=\left\{ (\eta, f)\in \Gamma_c(\FF^*\times \RR^k): \substack{ \mbox{$d_\FF \eta=0$, $d_\FF f=0$, $\mathrm{im}(\eta, f)\subset U$, }\\
	\mbox{ $s\circ \varphi\circ (\eta, f)$ and $t\circ \varphi\circ (\eta, f)$ are diffeo's} } \right\}
\]
is the image of some chart $(\VV, \Phi)$ if $\Diff(M, \pi)$ as described in Section \ref{sec:bisectionsliegroup}. Now, a section $(\eta, f)$ has image in $\operatorname{Rel}(\FF)$ if and only if $f=0$. It follows that $\Phi$ restricts to $\Fol(M, \pi)$ as 
\[
\Fol(M, \pi)\cap \mc{V} \xrightarrow{\Phi} \Omega^1_{\mathrm{cl}, c}(\FF)\cap \mc{U}. 
\]
Hence, the chart $(\mc{V}, \Phi)$ is adapted to $\Fol(M, \pi)$. 
\end{proof}

\begin{remark} Without the presence of the 2-form in a cosymplectic structure, our approach remains valid. Indeed, let $\alpha\in \Omega^1(M)$ be a closed nowhere vanishing one-form, and set $\FF=\ker \alpha$. A bisection $\sigma:M\to M\times M$ corresponds to a one-form preserving diffeomorphism if and only if $\sigma^*(\rar{\alpha})=\sigma^*(\lar{\alpha})=\alpha$. Furthermore, it corresponds to a foliation preserving diffeomorphism if and only if it is a foliated map $\sigma:(M, \FF)\to (M\times M, \FF\times \FF)$. The Moser lemma (\cite{smilde2021linearization}, Lemma 3.1) can be rephrased without the two-form. Using these observations, we can argue similar as in the proof of Theorem \ref{thm:cosymplecticliegroup} to obtain charts of $\Diff(M)$ adapted to $\Diff(M, \alpha)$ and $\Diff(M, \FF)$. The result is the following theorem. 
\begin{theorem}
	Let $\alpha_1, \dots, \alpha_k\in \Omega^1(M)$ be closed, everywhere independent one-forms and set $\FF=\cap_{i=1}^k\ker\alpha_i$. Then $\Diff(M, \FF)$ and $\Diff(M, \alpha_1, \dots, \alpha_k)$ are regular embedded Lie subgroups of $\Diff(M)$, with Lie algebras \begin{align*}
	\mf{X}_c(M, \FF)&=\{X\in \mf{X}_c(M) : [X, \Gamma(\FF)]\subset \Gamma(\FF)\} \mbox{ and }\\\mf{X}_c(M, \alpha_1, \dots, \alpha_k)&=\{ X\in \mf{X}_c(M): \LL_X(\alpha_1)=\dots=\LL_X(\alpha_k)=0\},
	\end{align*} respectively. 
\end{theorem}
\end{remark}

\appendix
\section{Linearization of Poisson groupoids over scattering-symplectic manifolds}\label{app:scattering}
Let $(M, Z, \omega)$ be a scattering-symplectic manifold, with underling Poisson structure $\pi$. If $\GG_{sc}\toto M$ is an integration of the scattering-tangent bundle ${}^{sc}T_ZM$, then it comes with a Poisson structure that is automatically linearizable by Theorem \ref{thm:dualintegrationlinearizable}, as the scattering-symplectic structure provides a Lie algebroid isomorphism ${}^{sc}T^*_ZM\cong {}^{sc}T_ZM$. However, from the perspective of the Poisson diffeomorphism group, the scattering algebroid is not the right one to consider, because the Poisson vector fields are sections of the log-tangent bundle, and not necessarily the scattering-tangent bundle.

Denote by $\pi_Z$ the Poisson structure on the log-tangent bundle $T_ZM$ induced by $\omega$:
\[
\begin{tikzcd}
	{}^{sc}T^*_Z M & {}^{sc} T_ZM \arrow[l, "\omega^\flat"'] \arrow[d]\\
	T^*_ZM \arrow[u] \arrow[r, "\pi_Z^\sharp"] & T_ZM.
\end{tikzcd}
\]

The section $\pi_Z$ degenerates over $Z$, so it is not symplectic anymore, and thus Theorem \ref{thm:dualintegrationlinearizable} does not apply to integrations of the bialgebroid $(T_ZM, T^*_ZM)$. Yet, those integrations are linearizable, as we will show in this section.
\begin{theorem}\label{thm:scatteringintegrationlinearizable}
	Let $(M, Z, \omega)$ be a scattering-symplectic manifold, with underlying $T_ZM$-Poisson structure $\pi_Z$. Let $\GG\toto M$ be an integration of the log-tangent bundle $T_ZM$. Then the Poisson groupoid $(\GG, \lar{\pi_Z}-\rar{\pi_Z})\toto M$ is linearizable.
\end{theorem}
\subsection{Linearization of scattering-symplectic structures}\label{sec:scatteringlagrangian} Let $(M, Z, \omega)$ be a scattering-symplectic manifold, with underlying Poisson structure $\pi$, and $i:L\hookrightarrow Z$ a Lagrangian submanifold transverse to $Z$. Certainly, the Poisson structure is linearizable around $L\setminus Z$ in $M\setminus Z$, by Weinstein's Lagrangian neighbourhood theorem. Moreover, the scattering symplectic structure induces a contact structure on $Z$ \cite{lanius2020} which is linearizable around $\tilde{Z}=Z\cap L$ (which is Legendrian in $Z$) by the Legendrian neighbourhood theorem in contact geometry. This gives an intuition to why the Poisson structure should be linearizable around $L$.

The first step in the linearization is a Moser lemma adapted to scattering-symplectic manifolds.
\begin{lemma}[$sc$-Moser lemma]\label{lem:scatteringmoser}
	Let $i:L\hookrightarrow (M, Z)$ be a submanifold transverse to $Z$. Let $\omega_0, \omega_1$ be two scattering symplectic forms such that $\omega_1\vert_L=\omega_0\vert_L$. Then there exists neighbourhoods $U_0, U_1$ of $L$ in $M$ and a diffeomorphism $\varphi:(U_0, U_0\cap Z)\to (U_1, U_1\cap Z)$ such that $\varphi^*\omega_1=\omega_0$ on $U_0$ and $\varphi\vert_L=\id_L$.
\end{lemma}
\begin{remark}
	For the scattering-symplectic forms $\omega_0, \omega_1$, the condition $\omega_1\vert_L=\omega_1\vert_L$ does \emph{not} imply that the cohomology classes in the scattering-cohomology agree. In fact, often they do not (see \cite{lanius2020}). The scattering-morphism changes the cohomology class of the symplectic form.
\end{remark}
\begin{proof}[Proof of Lemma \ref{lem:scatteringmoser}]
	By the splitting theorem for Lie algebroids (\cite{Bursztyn2016Splitting}, Theorem 4.1), applied to the log-tangent bundle, we can assume that $M=E$ is a vector bundle over $L$ and the hypersurface $Z$ is the bundle $E\vert_{\tilde{Z}}$ over the hypersurface $\tilde{Z}=Z\cap L$ in $L$. We set $\omega_t=\omega_0+t(\omega_1-\omega_0)$. 
	
	On $M\setminus Z$, we can write $\omega_1-\omega_0=d\psi$ with
	\[
	\psi=\int_0^1 \frac{1}{t}m^*_t(\iota_\EE (\omega_1-\omega_0)) dt.
	\]
	Here, $\EE$ is the Euler vector field of $E$, and $m_t$ is scalar multiplication by $t$. On $Z$, this is just a formal expression. We will show that there is a family of log-vector fields $X_t\in \mf{X}(M, Z)$ such that 
	\begin{equation}\label{eq:oneform}\tag{$*$}
		\iota_{X_t}\omega_t=-\psi \quad \mbox{on $M\setminus Z$}.
	\end{equation}
	By non-degeneracy of $\omega_t$ away from $Z$, the family of vector fields $X_t$ is uniquely determined on $M\setminus Z$, and satisfies $X_t\vert_{L\setminus\tilde{Z}}=0$. 
	
	To show that this extends to a log-vector field over $Z$, we need to zoom in on a neighbourhood of $Z$. Choose a defining function $\tilde{x}$ of $\tilde{Z}$ in $L$. Then $x=p^*\tilde{x}$ is a defining function for $Z$, with the nice property that $m^*_t(x)=x$ (and thus $\iota_\EE dx=0$). The key observation is that if $Y$ is a scattering vector field, then $Y/x$ is a log-vector field.
	\begin{claim*}
		The form $x\cdot \iota_\EE(\omega_1-\omega_0)$ is a scattering one-form.
	\end{claim*} 
	\begin{proof}[Proof of the claim]
		The Euler-vector field $\mc{E}$ is a log-vector field, and thus $x\mc{E}$ is a scattering-vector field. Therefore, $x\cdot \iota_\EE(\omega_1-\omega_0)=\iota_{x\mc{E}}(\omega_1-\omega_0)$ is a scattering one-form.
	\end{proof}
	From the claim, it follows that
	\[
	x\cdot \psi=x\int_0^1 \frac{1}{t}m^*_t(\iota_\EE (\omega_1-\omega_0)) dt=\int_0^1 \frac{1}{t} m^*_t(x\cdot\iota_\EE (\omega_1-\omega_0)) dt
	\]
	is a scattering form ($m_t$ is a log-map and thus a scattering algebroid morphism by Proposition \ref{prop:rescalingmorphisms}). Hence, we can find scattering vector fields $\tilde{X}_t$ near $Z$ such that  $\iota_{\tilde{X}_t}\omega_t=-x\cdot\psi$ in a neighbourhood of $Z$. The log-vector fields $X_t=\tilde{X}_t/x$ extend the solution to (\ref{eq:oneform}) over $Z$.
	
	Let $\varphi$ be the time-1 flow of $X_t$, which is defined in a neighbourhood $U$ of $L$ because $X_t\vert_L=0$. Then $\varphi$ is a $b$-map that satisfies $\varphi^*\omega_1=\omega_0$ outside $Z$ by the standard Moser argument. But since both are scattering forms, they must agree also on $Z$.
\end{proof}

\subsubsection{Scattering Lagrangian neighbourhood theorem}
Let $i:(L, \tilde{Z})\hookrightarrow (M, Z, \omega)$ be a Lagrangian submanifold transverse to $Z$. As in the proof of Lemma \ref{lem:scatteringmoser}, we can assume that $(M, Z)=(NL, NL\vert_{\tilde{Z}})$. We can already linearize $\omega$ on the scattering algebroid level.
\begin{lemma}\label{lem:scatteringlinearization}
	The limit
	\[
	\omega_{\mathrm{lin}}=\lim_{t\to 0} \frac{m^*_t\omega}{t}
	\]
	defines a linear scattering 2-form on $(NL, NL\vert_{\tilde{Z}})$ that satisfies $\omega_{\mathrm{lin}}\vert_L=\omega\vert_L$.
\end{lemma}
\begin{proof}
	For every $t\in \RR$, the map $m_t:NL\to NL$ is a log-morphism and thus a scattering morphism by Proposition \ref{prop:rescalingmorphisms}. Therefore, $m^*_t\omega$ is scattering 2-form for all $t\in \RR$. Because $L$ is Lagrangian, $m_0^*\omega=0$. Since $m_t^*\omega$ is smooth in $t$, the limit $\lim_{t\to 0} m^*_t\omega/t$ exists and the resulting form is again smooth. Clearly, $m^*_\lambda \omega_{\mathrm{lin}}=\lambda \omega_{\mathrm{lin}}$, so the resulting form is indeed linear.
\end{proof}

\begin{theorem}[$sc$-Lagrangian neighbourhood theorem]\label{thm:sclagrangianneighbourhood}
	Let $(M, Z, \omega)$ be a scattering symplectic manifold with Poisson structure $\pi$ and $i:L\hookrightarrow (M, Z)$ a Lagrangian submanifold transverse to $Z$. Then there are neighbourhoods $U$ of $L$ in $M$ and $V$ of $L$ in $NL$ and a Poisson diffeomorphism on $\varphi:(U, \pi)\to (V, \pi_{\mathrm{lin}})$ that restricts to the identity on $L$.
\end{theorem}
\begin{proof}
	The theorem follows from Lemma \ref{lem:scatteringmoser} and Lemma \ref{lem:scatteringlinearization}.
\end{proof}

\subsection{Poisson groupoids over scattering-symplectic manifolds}

Let $(M,Z)$ be a smooth log-manifold, with $p_Z:T_ZM\to M$ and $p_{sc}:{}^{sc}T_ZM\to M$ the bundle projections. The scattering Lie algebroid ${}^{sc}T_ZM$ is naturally anchored to the log-tangent bundle $T_ZM$, and therefore, according to \cite{smilde2021linearization}, if $\GG\toto M$ is an integration of $T_ZM$, the following squares are a double Lie algebroid and an LA-groupoid, respectively.
\[
\begin{tikzcd}
	p_{sc}^!T_ZM \arrow[d, Rightarrow] \arrow[r, Rightarrow] & T_ZM \arrow[d, Rightarrow]\\
	{}^{sc}T_ZM \arrow[r, Rightarrow] & M,
\end{tikzcd}
\quad
\begin{tikzcd}
	p_{sc}^!\GG \arrow[d, shift left] \arrow[d, shift right] \arrow[r, Rightarrow] & \GG \arrow[d, shift left]\arrow[d, shift right] \\ 
	{}^{sc}T_ZM \arrow[r, Rightarrow] & M.
\end{tikzcd}
\]
\begin{proposition}\label{prop:pullbackscatteringb}
	The pair $(T_ZM, T_ZM\vert_Z)$ is a log-manifold. Furthermore:
	\begin{itemize}[noitemsep, topsep=0em]
		\item[(i).] the prolongation algebroid $p^!_Z (T_ZM)\to T_ZM$ is isomorphic to the log-tangent bundle $T_{(T_ZM\vert_Z)} (T_ZM)$,
		\item[(ii).] the prolongation algebroid $p^!_{sc}(T_ZM)\to T_ZM$ is isomorphic to the scattering algebroid ${}^{sc}T_{(T_ZM\vert_Z)} (T_ZM)$.
	\end{itemize}
	If $\GG\toto M$ is an integration of $T_ZM$, then, with $\mc{H}=s^{-1}(Z)=t^{-1}(Z)$ as a critical hypersurface, the following hold:
	\begin{itemize}[noitemsep, topsep=0em]
		\item[(iii).] the Lie algebroid $p^!_Z\GG\Rightarrow \GG$ is isomorphic to the log-tangent bundle $T_{\mc{H}}\GG$.
		\item[(iv).] the Lie algebroid $p^!_{sc}\GG\Rightarrow \GG$ is isomorphic to the scattering-tangent bundle ${}^{sc}T_{\mc{H}}\GG$. 
	\end{itemize}
\end{proposition}
\begin{proof}
	We follow the notation in \cite{smilde2021linearization}, Section 2.
	
	Recall that $p^!_Z(T_ZM)$ has side bundles and core isomorphic to $T_ZM$. The same is true for $T_{(T_ZM\vert_Z)}(T_ZM)$. 
	
	The sections of $p^!_Z(T_ZM)\to T_ZM$ are generated by the complete lifts $\widetilde{v}$ and vertical lifts $v^\uparrow$ for $v\in \Gamma(T_ZM)$. Since the $\varphi_t$ of a section $v\in \Gamma(T_ZM)$ preserves $T_ZM\vert_Z$, it follows that $\widetilde{v}$ is tangent to $T_ZM\vert_Z$, and thus a section of $T_{(T_ZM\vert_Z)}(T_ZM)$. This shows that the anchor of $p^!_Z (T_Z M)\to T_Z M$ factors through $T_{(T_ZM\vert_Z)}(T_ZM)$. The map $p^!_Z(T_Z M)\to T_{(T_ZM\vert_Z)}(T_ZM)$ is a map of double vector bundles, because it sends linear sections to linear sections and core sections to core sections. Now it follows that the map is an isomorphism as it is the identity on the side bundles and the core. 
	
	For the scattering bundles, the argument is similar. The sections of $p_{sc}^!(T_ZM)\to T_ZM$ are generated by $T_Zv$ and the core lifts $\widehat{v}$ for $v\in \Gamma({}^{sc}T_ZM)$.  The anchor sends $T_Z v$ to the vector field $\widetilde{v}$, which vanishes (as a log-vector field) over $T_ZM\vert_Z$ (indeed, the flow of scattering vector fields induce the identity on $T_ZM\vert_Z$), and is therefore a scattering-vector field. The anchor sends $\widehat{v}$ to $v^\uparrow$, which too vanishes as a log-vector field over $T_ZM\vert_Z$. It follows that there is an induced map of double vector bundles $p_{sc}^!(T_ZM)\to {}^{sc}T_{(T_ZM\vert_Z)} (T_ZM)$ that is an isomorphism because it is the identity on the side bundles and the core.
	
	The proofs of statements \textit{(iii)} and \textit{(iv)} are similar. One replaces $T_Z v$ by the star sections $(t^*v+s^*v, v)$ and the core sections by the right-invariant core sections $t^*v$ (whose anchor is the right-invariant vector field $\overrightarrow{v}$ on $\GG$) as in the proof of Theorem 2.16 in \cite{smilde2021linearization}. 
\end{proof}
\begin{proof}[Proof of Theorem \ref{thm:scatteringintegrationlinearizable}] 
Let $(M, Z, \omega)$ be a scattering-symplectic manifold and $\GG\toto M$ an integration of the log-tangent bundle $T_ZM$. In the notation of Proposition \ref{prop:pullbackscatteringb}, let $\tilde{s}, \tilde{t}:p^!_{sc}\GG\to {}^{sc}T_ZM$ be the source- and target map. Then $(\GG, \mc{H}, \tilde{s}^*\omega-\tilde{t}^*\omega)$ is a scattering-symplectic manifold with underlying Poisson structure $\lar{\pi_Z}-\rar{\pi_Z}$ on $\GG$, where $\pi_Z\in \Gamma(\wedge^2T_ZM)$ is the Poisson structure on the log-tangent bundle induced by $\omega$. Since $M\subset \GG$ is transverse to the hypersurface $\mc{H}$, the Poisson groupoid $(\GG, \Pi)$ is linearizable around $M$ by the scattering-Lagrangian neighbourhood Theorem \ref{thm:sclagrangianneighbourhood}. This completes the proof.
\end{proof}